\documentclass[a4paper,fleqn]{amsart}
\pdfoutput=1

\usepackage[utf8]{inputenc}
\usepackage[T1]{fontenc} 
\usepackage{microtype} 
\usepackage{lmodern}   

\usepackage[english]{babel} 

\usepackage[hmargin=3cm, vmargin=4cm]{geometry} 
\usepackage[hang, small,labelfont=bf,up,textfont=it,up]{caption} 
\usepackage{booktabs} 

\usepackage[shortlabels]{enumitem} 
\setlist[itemize]{noitemsep} 

\usepackage{hyperref} 

\usepackage{import}
\usepackage{xifthen}
\usepackage{pdfpages}
\usepackage{transparent}
\usepackage{wrapfig}

\usepackage{subcaption}

\usepackage[mode=build]{standalone}
\usepackage{tikz}
\usetikzlibrary{cd}
\usetikzlibrary{calc}
\usetikzlibrary{decorations.markings}
\tikzstyle{state}=[
rectangle,
minimum size =1.25cm,
thick,
align=center
]
\usepackage{scalefnt}

\usepackage[
style=apa,
backend=biber,
alldates=year,
sortcites=false,
]{biblatex}
\usepackage{xcolor, soulutf8}
\definecolor{bubbles}{rgb}{0.91, 1.0, 1.0}
\definecolor{aquamarine}{rgb}{0.5, 1.0, 0.83}
\definecolor{bubblegum}{rgb}{0.99, 0.76, 0.8}
\definecolor{bluebell}{rgb}{0.64, 0.64, 0.82}
\definecolor{dollarbill}{rgb}{0.72, 0.93, 0.6}

\usepackage{amssymb}
\usepackage{amsmath}
\usepackage{makecell}

\usepackage{amsthm}
\newtheorem{theorem}{Theorem}
\newtheorem{remark}[theorem]{Remark}
\newtheorem{lemma}[theorem]{Lemma}

\newcommand{\figref}[1]{Fig.~\ref{#1}}

\newcommand{\Z}{\mathbb Z}
\newcommand{\Q}{\mathbb Q}

\newcommand{\C}{\mathbb C}
\newcommand{\Proj}{\mathbb P}
\newcommand{\im}{ \operatorname{im}}
\newcommand{\res}{\operatorname{Res}}
\newcommand{\NS}{\operatorname{NS}}

\newcommand{\PH}{\mathit{PH}}

\newcommand{\Hdr}{H_\mathrm{DR}}
\newcommand{\PHdr}{\PH_\mathrm{DR}}
\newcommand{\ud}{\mathrm{d}}
\def\st{\ \middle|\ }
\newcommand{\eqdef}{\smash{\overset{\small{\mathrm{def}}}{=}}}

\title[Effective homology and periods]{Effective homology and periods\\ of complex projective hypersurfaces} 

\author{Pierre Lairez}
\address{Université Paris-Saclay, Inria, 91120 Palaiseau, France}
\email{pierre.lairez@inria.fr}

\author{Eric Pichon-Pharabod}
\address{Institut de Physique Theorique, Université Paris-Saclay, CEA, CNRS, F-91191 Gif-sur-
  Yvette Cedex, France; Université Paris-Saclay, Inria, 91120 Palaiseau, France}
\email{eric.pichon@polytechnique.edu}

\author{Pierre Vanhove}
\address{Institut de Physique Theorique, Université Paris-Saclay, CEA, CNRS, F-91191 Gif-sur-
  Yvette Cedex, France}
\email{pierre.vanhove@ipht.fr}

\date{\today}

\thanks{This work has been supported by the Agence nationale de la recherche
  (ANR), grant agreement ANR-19-CE40-0018 (De Rerum Natura), grant agreement
ANR-20-CE40-0026-01 (Symmetries and moduli spaces in algebraic geometry and physics); and the European
  Research Council (ERC) under the European Union’s Horizon Europe research and
  innovation programme, grant agreement 101040794 (10000~DIGITS)}


\begin{document}

\maketitle

\begin{abstract}
  We introduce a new algorithm for computing the periods of a smooth complex projective hypersurface.
  The algorithm intertwine with a new method for computing an explicit basis of the singular homology of the hypersurface.
  It is based on Picard--Lefschetz theory and relies on the computation of the monodromy action induced by a one-parameter family of hyperplane sections on the homology of a given section.

  We provide a SageMath implementation. For example, on a laptop, it makes it possible to compute the periods of a smooth complex quartic surface
  with hundreds of digits of precision in typically an hour.
\end{abstract}


\setcounter{tocdepth}{1}
\tableofcontents

\section{Introduction}
The $k$-th period matrix of a smooth complex variety~$X$ is the matrix of the
De~Rham duality~$H_k(X) \times \Hdr^k(X) \to \mathbb{C}$
\begin{equation}\label{eq:21}
  \left( [\gamma], [\omega] \right) \mapsto \int_\gamma \omega,
\end{equation}
between singular homology and (algebraic) De~Rham cohomology.
When~$\Hdr^k(X)$ is endowed with additional structure coming from Hodge theory, we obtain a remarkable
continuous invariant of~$X$ reflecting an interplay between the complex algebraic structure and the topological one \parencite{Griffiths_1968a,CarlsonMullerStachPeters_2017}.
In experimental mathematics, the numerical computation of periods is a useful tool to investigate some algebraic invariants \parencite{ChudnovskyChudnovsky_1989}.
Periods linearise some aspects of algebraic varieties, at the cost of introducing transcendental functions, with the integral sign. This transcendental nature makes it difficult to compute exactly with periods, but large precision (typically hundreds or thousands of decimal digits) may allow to recover exact invariants.

For curves, the numerical computation of periods gives access to an approximate representation of the Jacobian, which in turn leads to interesting invariants.
This has been used to compute some of the data related to genus~2 curves
in the LMFDB \parencite{BookerSijslingSutherlandVoightYasaki_2016}, such as the endomorphism rings \parencite{CostaMascotSijslingVoight_2019} or the Sato--Tate groups \parencite{FiteKedlayaRotgerSutherland_2012}, which can be obtained by recovering integer relations between the periods.
For surfaces, periods also lead to interesting invariants that are hard to compute by other means.
For example, for an algebraic surface~$X$, Lefschetz's $(1,1)$-theorem \parencite[p.~163]{GriffithsHarris_1978} relates the integer relations between the periods of~$X$ with the algebraic curves lying on~$X$.
Integer relations can be recovered from numerical approximations of the periods thanks to lattice reduction algorithms \parencite{LenstraLenstraLovasz_1982}.
This leads to a practical algorithm for computing, heuristically, the Néron--Severi group of a surface in~$\mathbb{P}^3(\mathbb{C})$ \parencite{LairezSertoz_2019}.
The appearance of explicit algebraic varieties in various fields of mathematics, such as  Diophantine approximation \parencite{BeukersPeters_1984} or mathematical physics \parencite{BlochKerrVanhove_2015},
is a strong incentive for developing automatic methods to compute algebraic invariants.

\subsection{Contributions}
We focus on the case of projective hypersurfaces, although most of the ideas
that we present apply in more generality.
Let~$X \subset \mathbb{P}^n(\mathbb{C})$ be a complex projective hypersurface defined by a polynomial~$F(x_0,\dotsc,x_n) \in \mathbb{C}[x_0,\dotsc,x_n]$.
A major obstacle in the computation of the period matrices for~$X$ is the lack of algorithms for computing an explicit description of the singular homology of~$X$.
By ``explicit'', we mean explicit enough to integrate numerically over a basis of cycles.
We describe an algorithm for computing at the same time:
\begin{enumerate}
  \item an explicit basis of the singular homology of~$X$;
  \item a numerical approximation, with rigorous error bounds, of the period
        matrix of~$X$ with respect to this homology basis, and the
        Griffiths--Dwork cohomology basis.
\end{enumerate}
With respect to precision only, the complexity of the algorithm is quasilinear: computing twice as much digits takes roughly twice as much time.
Note that, as a specificity of projective complete intersections,
only the $n$-th period matrix of~$X$ is interesting, the other ones are trivial ($1\times 1$ or~$0\times 0$ matrices).

The algorithm relies on the study of a one-parameter family of hyperplane sections~$X \cap H_t$,
following the principles of Picard--Lefschetz theory \parencite{Lefschetz_1924,Lamotke_1981}.
An important step of the algorithm is the computation of the monodromy action induced by this family of sections on the homology~$H_{n-1}(X \cap H_b)$ of one section (above a chosen basepoint $b\in \Proj^1$).
We perform this computation through duality with De~Rham cohomology,
using the period matrix of~$X \cap H_b$, obtained by induction on dimension.

To perform the numerical integration, we reduce to one-variable complex integrals
of rational multiples of the periods of~$H_{n-1}(X\cap H_b)$.
Since these periods are solutions of a Picard--Fuchs differential equation, which we can compute explicitly, we can perform the integration using general algorithms for integration of \emph{differentially finite functions} \parencite{VanDerHoeven_1999,Mezzarobba_2010}. 
These algorithms provide rigorous error bounds with quasilinear complexity with respect to precision.

On the practical side, we have implemented the above algorithm in Sagemath \parencite{sagemath}, together with the computation of related invariants (Picard rank, Néron--Severi group, endomorphism ring).
This implementation can be found in the package \mbox{\emph{lefschetz-family}}\footnote{\url{https://gitlab.inria.fr/epichonp/lefschetz-family}}.
We are able to compute the periods of the holomorphic form of a smooth quartic surface in~$\mathbb{P}^3(\mathbb{C})$ defined over~$\mathbb{Q}$ in, typically, about an hour on a laptop, with 300~decimal digits of precision. Naturally, the actual running time depends on many parameters, including the bitsize of the coefficients of the defining equation, and the conditionning of some numerical steps, see Section~\ref{sec:experiments} for concrete examples.
Such a computation was not feasible in reasonable time with the previously known algorithm \parencite{Sertoz_2019}, except for some quartic surfaces defined by sparse polynomials.

At the moment, our algorithm only applies to smooth projective hypersurfaces
but the methodology has a wider scope. To illustrate it,
we explain the computation of the periods of a $K3$ surface described as the desingularisation of a nodal quartic surface, using some \emph{ad hoc} ingredients (Section~\ref{sec:an-example-from}).

\subsection{Related works}
Algorithms for computing period matrices of curves (also known as Riemann matrices)
are well established, with work by~\textcite{DeconinckVanHoeij_2001,Swierczewski_2017,BruinSijslingZotine_2019,MolinNeurohr_2019,Neurohr_2018}, to name a few.
The papers by~\textcite{CynkvanStraten2019,ElsenhansJahnel2022} are the first to tackle higher dimensions in some particular cases: double covers of $\mathbb{P}^3(\mathbb{C})$ ramified along 8~planes,
and double covers of~$\mathbb{P}^2(\mathbb{C})$ ramified along 6~lines, respectively.
The algorithm by~\textcite{Sertoz_2019}
is the first and only algorithm to tackle the case of smooth hypersurfaces in any dimension.
It proceeds by global deformation from the hypersurface~$X \subset \mathbb{P}^n(\mathbb{C})$ to the Fermat hypersurface defined by $\sum_{i=0}^n x_i^d = 0$, with~$d$ equal to the degree of~$X$.
In a nutshell, our method is more efficient, in most cases, because the Picard--Fuchs equation associated to this global deformation is generally much larger than the Picard--Fuchs differential equation associated to a family of hyperplane sections.
Moreover, Sertöz' method requires a \emph{base point}, a hypersurface of which we know the periods \emph{a priori}.
Our method is more intrinsic, which gives it the potential to study varieties beyond smooth projective hypersurfaces.

Computing a period matrix of an algebraic variety requires to address, in one way or another, three computational problems reflecting the very definition of a period matrix:
first, compute a basis of the homology, second, compute a basis of the cohomology, and, third, compute the coefficients of the pairing of~\eqref{eq:21} by numerical integration.
Each of these problems is related to a different research field.

\subsubsection{Singular homology}
Computing a basis of the singular homology is a major obstacle. While there are algorithms for computing the homology of submanifolds of~$\mathbb{R}^n$ from sample points \parencite{NiyogiSmaleWeinberger_2008,CuckerKrickShub_2018} they seem challenging to implement and exploit. A complex surface in~$\mathbb{P}^3(\mathbb{C})$ is already a 4-dimensional real manifold in a 6-dimensional ambient space.
Recent experiments suggest that the number of samples required to compute rigorously the homology is rather large \parencite{DiRoccoEklundGafvert_2020} compared to the Betti numbers. 
Furthermore, we do not know how to use such a data structure (homology computed from sample points) to efficiently compute periods.

In the case of curves, the problem reduces to computing the monodromy action on the roots of a univariate polynomial depending on a parameter \parencite{TretkoffTretkoff_1984}.
In the case of projective hypersurfaces, Sertöz' algorithm deals with the problem indirectly
and only need a description of the homology of Fermat hypersurface, which has been worked out by \parencite{Pham_1965}.

\subsubsection{De Rham cohomology}

The algebraic De Rham cohomology,
in the case of projective hypersurfaces,
is well described
by the Griffiths--Dwork reduction \parencite{Griffiths_1969a}, both theoretically and computationally (see Section~\ref{sec:de-rham-middle}).
There are also algorithms for the affine case \parencite{OakuTakayama_1999},
but we are not aware of any algorithm for the general case of a projective variety (not to mention that there are also projective varieties that are not naturally embedded in a projective space, like multiprojective varieties, or fibrations). 
This is the main reason for restricting our focus to projective hypersurfaces.

\subsubsection{Numerical integration}
For the integration step, a direct approach seems possible. If the homology
basis is sufficently explicit, and if we can numerically evaluate the
differential forms defining the cohomology basis at any given point, we can
certainly compute the pairing of Equation~\eqref{eq:21}. 

However we aim for high precision, so all
finite-order quadrature methods (like Simpson's rule or Monte-Carlo algorithms) are ruled out because they have exponential complexity with respect to precision. The periods are
$n$-dimensional integrals of algebraic functions, where~$n$ is the complex dimension of~$X$. 
Assuming that we can formulate them as integrals over a $n$-simplex, or~$[0,1]^n$, we can compute them with the
Gauss--Lengendre quadrature formula. 
We expect a~$p^{n+1 + o(1)}$ complexity with respect to the precision~$p$. 
For computing the periods of curves, this is a $p^{2+o(1)}$ complexity and this method is the most commonly used.
In our method, we use numerical analytic continuation of differentially finite functions
to compute integrals. The complexity is quasilinear with respect to precision, $p^{1+o(1)}$, thanks to binary splitting.
Numerical analytic continuation is also what we use to compute monodromy actions, so it is natural to use it for integration as well.
Note however that, in our typical computations (involving large degree differential operators and large but not extreme precision), we are far from the threshold where binary splitting gets better than $p^{2+o(1)}$ methods.

\subsection{Content}

Let $X \subset \mathbb{P}^n(\mathbb{C})$ be a smooth hypersurface variety.
Let~$(H_t)_{t\in \mathbb{P}^1}$ be a generic pencil of hyperplanes.
Up to a change of coordinates, we may assume that~$H_t = V(x_n - tx_0)$.
Let~$b \in \mathbb{P}^1$ be a generic base point.
Let~$\Sigma \subset \mathbb{P}^1$ be the set of all~$t$ such that~$X\cap H_t$ is singular.
The pencil induces an action of~$\pi_1(\mathbb{P}^1\setminus \Sigma)$ on~$H_{n-1}(X \cap H_b)$.
In Section~\ref{sec:effect-picard-lefsch}, we show, mostly following~\textcite{Lamotke_1981},
how this monodromy action determines the homology of~$X$ in an explicit way.
Section~\ref{sec:algorithms} decribes the main algorithm. Section~\ref{example} details the example of a quartic surface.
Section~\ref{sec:further-algorithms} deals with the computation of the intersection product on~$H_2(X)$, when~$X$ is a smooth surface in~$\mathbb{P}^3$.
Section~\ref{sec:experiments} reports on benchmarks and applications.

\subsection{Acknowledgments}
We thank Marc Mezzarobba for his help with interfacing with the \emph{ore\_algebra} package.
The examples of Sections~\ref{Bouyer} and~\ref{symmetric}, the random search of Section~\ref{sec:picard_rank_examples} and the comparisons of Section~\ref{sec:comparison} were run on the Infinity Cluster hosted by Institut d'Astrophysique de Paris.  
We thank Stephane Rouberol for running smoothly this cluster for us.
We thank Xavier Roulleau for bringing the works of Keiji Oguiso to our attention.
We also thank Alice Garbagnati, Erik Panzer and Emre Sertöz for interesting and enlightening discussions and comments.
Finally, we thank the referees for their careful reading of the manuscript, and their valuable comments and suggestions.

\section{Picard-Lefschetz theory}
\label{sec:effect-picard-lefsch}

\subsection{Lefschetz fibrations and Lefschetz' main lemma}

Following~\textcite{Lamotke_1981}, we recall a few elements of Picard--Lefschetz theory that are important to our study.

\subsubsection{Lefschetz fibrations}
\label{sec:lefschetz-fibrations}

Let $\Proj^N$ denote the $N$-dimensional complex projective space. Let $A\subset \Proj^N$ be an $N-2$-dimensional projective subspace.
The {\em pencil of axis $A$} is the one-dimensional family of hyperplanes of $\Proj^N$ that contain $A$.
It is parametrised by $\Proj^1$. Concretely, if~$A$ is the vanishing locus of two linear forms~$\lambda$ and~$\mu$,
then for~$t \in \mathbb{P}^1$, we define~$H_t = V(\lambda - t \mu)$ (and~$H_\infty = V(\mu)$).

We consider an irreducible, closed complex projective variety $X\subset \Proj^N$, with dimension $\dim X = n$.
We aim at studying the topology of~$X$ through the intersections of~$X$ with the hyperplanes~$H_t$.
For~$t\in \mathbb{P}^1$, let
\begin{equation}
  X_t \eqdef X\cap H_t.
\end{equation}
For any~$x \in X\setminus A$, there is a unique~$t \in \mathbb{P}^1$ such that~$x \in H_t$.
This defines a rational map~$X \dashrightarrow \mathbb{P}^1$.
It is useful to consider a modification~$Y$ of~$X$:
\begin{equation}
  \label{eq:18}
  Y \eqdef \{(x,t)\in X\times \mathbb{P}^1\mid x\in H_t\}.
\end{equation}
The projection on the first coordinate induces a proper map~$\pi:Y\to X$,
which is an isomorphism above~$X\setminus A$.
The projection on the second coordinate induces a regular map~$f : Y\to \mathbb{P}^1$
which resolves the indeterminacies of the map~$X \dashrightarrow \mathbb{P}^1$.
The map $\pi$ is the blowup of~$X$ at the base locus~$X' \eqdef X\cap A$ of $X \dashrightarrow \mathbb{P}^1$.
The fibre~$f^{-1}(t)$ above some $t\in \mathbb{P}^1$ is isomorphic to~$X_t$.
Indeed, by definition, $f^{-1}(t) = \left\{ x\in X \st x\in H_t \right\} = X\cap H_t$.

Let $\Sigma \eqdef \left\{ f(x) \ \middle|\ x\in Y\text{ and } \mathrm df = 0 \right\}$ be the set of critical values.
It is also the set of all~$t\in \mathbb{P}^1$ such that~$X_t$ is singular.
It is well known that $\Sigma$ is finite (since~$f$ is locally constant on the subvariety~$\left\{ \mathrm df = 0 \right\}\subseteq Y$).
If the critical points of~$f$ are all nondegenerate (meaning the Hessian matrix at the critical point is nonsingular) and the critical values associated to different critical points are distinct, this is a {\em Lefschetz fibration}.
We will work under this hypothesis, which is satisfied as long as~$A$ is generic \parencite[\S 1.6]{Lamotke_1981}.


\subsubsection{Monodromy and extensions}
\label{sec:monodromy-extensions}

We recall the concepts of monodromy and extensions \parencite[see][\S 6.4]{Lamotke_1981}.
By Ehresmann's theorem, the restriction of $f$ to $f^{-1}(\mathbb{P}^1\setminus \Sigma)$ is a smooth fibre bundle:
for any simply connected open set~$U\subseteq \mathbb{P}^1\setminus \Sigma$ and any~$b\in U$,
there is a diffeomorphism~$\Phi_U : f^{-1}(U) \to X_b \times U$
with the compatibility $\operatorname{pr}_2 \circ \Phi_U = f|_{f^{-1}(U)}$
and~$\operatorname{pr}_1 \circ \Phi_U |_{X_b} = \operatorname{id}_{X_b}$.

In particular, for any continuous path $\ell: [0,1]\to \mathbb{P}^1\setminus \Sigma$
without self intersection
from a point~$b = \ell(0)$ to another~$b'= \ell(1)$,
we obtain a continuous deformation of~$X_b$ into~$X_{b'}$.
Namely, we pick a simply connected neighbourhood~$U$ of~$\ell([0,1])$
and we have the induced diffeomorphism
\[ X_b \to X_{b'}, \quad x \mapsto \Phi_U^{-1}( x, b'). \]
This diffeomorphism is uniquely determined up to homotopy
so it induces a uniquely determined isomorphism~$\ell_* : H_q(X_b) \to H_q(X_{b'})$, for any~$q$.
This map depends only on the homotopy class of~$\ell$ in~$\mathbb{P}^1\setminus \Sigma$.
This is the {\em action of monodromy along $\ell$ on homology}.
To extend this notion to self-intersecting paths,
we cut the paths into non-self-intersecting pieces
and compose the actions of each piece.
For any paths~$\ell$ and~$\ell'$, with~$\ell(1) = \ell'(0)$,
let $\ell'\ell$ denote the composition (go through~$\ell$ then~$\ell'$).
Then~$(\ell' \ell)_* = \ell'_* \ell_*$.
This defines the monodromy action of~$\pi(\mathbb{P}^1\setminus \Sigma)$ on~$H_q(X_b)$.

\begin{figure}[tbp]
  \centering
  \begin{center}
    \resizebox{.7\linewidth}{!}{
  \resizebox{\columnwidth}{!}{\begin{tikzpicture}
	\begin{scope}
        	\coordinate (A) at (0,0);
        	\coordinate (B) at (1,0);
        	\coordinate (C) at (1,0);
        	\coordinate (D) at (2,0.5);
        	\coordinate (E) at (3,1);
        	\coordinate (F) at (3,1);
        	\coordinate (G) at (4,1);
        	
        	\coordinate (A2) at (0,2);
        	\coordinate (B2) at (1,2);
        	\coordinate (C2) at (1,2);
        	\coordinate (D2) at (2,2.5);
        	\coordinate (E2) at (3,3);
        	\coordinate (F2) at (3,3);
        	\coordinate (G2) at (4,3);
        	
        	\draw [fill opacity=0.1,fill=orange, color=orange]  (A) ..controls (B) and (C).. (D)  ..controls (E) and (F).. (G) arc (270:90:0.5 and 1) 
        	-- (G2) ..controls (F2) and (E2).. (D2)  ..controls (C2) and (B2).. (A2)
        	 arc (90:270:0.5 and 1) ;
	
	\draw [fill opacity=0.1,fill=orange, color=orange]  (A) ..controls (B) and (C).. (D)  ..controls (E) and (F).. (G) 
        	 arc (-90:90:0.5 and 1) 
        	-- (G2) ..controls (F2) and (E2).. (D2)  ..controls (C2) and (B2).. (A2)
        	 arc (90:-90:0.5 and 1) ;
	
        \draw (2,1.5) node[color=orange] {$\tau_\ell(\gamma)$};
        
        \draw [fill opacity=0.1,fill=orange, color=orange]  (8,0) ..controls (7,0) and (7,0).. (6,0.5)  ..controls (5,1) and (5,1).. (4,1) 
        	 arc (-90:90:0.5 and 1) 
        	-- (4,3) ..controls (5,3) and (5,3).. (6,2.5)  ..controls (7,2) and (7,2).. (8,2)
        	 arc (90:-90:0.5 and 1) ;
	
	\draw [fill opacity=0.1,fill=orange, color=orange]  (8,0) ..controls (7,0) and (7,0).. (6,0.5)  ..controls (5,1) and (5,1).. (4,1) 
        	 arc (270:90:0.5 and 1) 
        	-- (4,3) ..controls (5,3) and (5,3).. (6,2.5)  ..controls (7,2) and (7,2).. (8,2)
        	 arc (90:270:0.5 and 1) ;

        \draw (6,1.5) node[color=orange] {$\tau_{\ell'}(\ell_*\gamma)$};
	 
        \draw [very thick, color=red] (0,1) ellipse (0.5 and 1);
        \draw (0,2.3) node[above, color=red] {$\gamma$};
        
        \draw [very thick , color=red] (4,2) ellipse (0.5 and 1);
        \draw (4,3.3) node[above, color=red] {$\ell_*\gamma$};
        
        \draw [very thick , color=red] (8,1) ellipse (0.5 and 1);
        \draw (8,2.3) node[above, color=red] {$\ell'_*\ell_*\gamma = (\ell'\ell)_*\gamma$};
        
        \draw [color=blue] (0,-2) node[color=black!40!green] {$\bullet$} node[above, color=black!40!green] {$b$} 
        ..controls (1,-2) and (2,-1).. (4,-1) node[midway, above] {$\ell$} node[midway, below, sloped] {$\longrightarrow$} 
        node[color=black!40!green] {$\bullet$} node[above, color=black!40!green] {$b'$}
        ..controls (5,-1) and (7,-2).. (8,-2) node[midway, above] {$\ell'$} node[midway, below, sloped] {$\longrightarrow$} 
        node[color=black!40!green] {$\bullet$} node[above, color=black!40!green] {$b''$};
        \end{scope}
\end{tikzpicture}}

    }
  \end{center}
  \caption{The monodromy and extensions of a cycle $\gamma$ along two composable paths $\ell$ and $\ell'$. 
  The extension of $\gamma$ along $\ell\ell'$ is the sum of the extensions of $\gamma$ along $\ell$ and of $\ell_*\gamma$ along $\ell'$. 
  As stated in \eqref{relation_extension_monodromy}, the border of $\tau_\ell(\gamma)$ is $\ell_*\gamma - \gamma$.}
  \label{extension_monodromy}
\end{figure}

In the same setting, given a $q$-chain~$A$ in~$X_b$,
we may consider the $q+1$-chain~$A \times [0,1]$ in~$X_b\times [0,1]$
and its image~$\Phi_U(A \times [0,1])$ in~$Y$.
Given any open set $V$ containing~$f^{-1}(U)$, this induces a map
\begin{equation}
  \tau_\ell : H_{q}(X_b) \to H_{q+1}(V, X_b\cup X_{b'})
\end{equation}
called the {\em extension along $\ell$}.
(This map also follows from the Künneth formula
for the pairs~$(X_b, \varnothing)$ and~$([0,1], \{0,1\})$ since~$f^{-1}(\ell([0,1])) \simeq X_b \times [0,1]$.)
The map~$\tau_\ell$ depends only on the homotopy class of~$\ell$.
Given two paths $\ell$ and $\ell'$ with~$\ell(1) = \ell'(0)$ we have (Fig.~\ref{extension_monodromy})
\begin{equation}
  \label{extension_composition}
  \tau_{\ell'\ell} = \tau_{\ell} + \tau_{\ell'}\circ \ell_*\,.
\end{equation}
Using this composition rule one may define extensions along paths with self intersections.
When the path~$\ell$ is a loop from~$b$ to~$b$, we obtain a map
\begin{equation}
  \tau_\ell : H_{q}(X_b) \to H_{q+1}(V, X_b),
\end{equation}
for any neighbourhood~$V \subseteq Y$ of~$f^{-1}(\ell([0,1]))$.

These two concepts, extension and monodromy, are related by the formula
\begin{equation}
  \label{relation_extension_monodromy}
  (-1)^{n-1} \partial \circ \tau_\ell = \ell_* -\operatorname{id}\,,
\end{equation}
where $\partial: H_{q}(Y, X_b\cup X_b') \to H_{q-1}(X_b)\oplus H_{q-1}(X_{b'})$ is the border map \parencite[\S 6.4.6]{Lamotke_1981}.


\subsubsection{Vanishing cycles}
\label{sec:vanishing-cycles}

\begin{figure}[tbp]
  \centering
  \resizebox{.65\linewidth}{!}{
  \resizebox{\columnwidth}{!}{\newcommand{\ellpath}[5] 
{
	\draw (#1) node {$\bullet$} node[above] {#3}  [decoration={markings, mark=at position 1 with {\arrow{>}}}, postaction={decorate}] circle[radius=1] ($(#1)+(1,1)$) node {#4};
	\draw (#2) to[bend right, #5]  node[midway, below, sloped] {$\longrightarrow$} node[midway, above, sloped] {$\longleftarrow$} ($(#1)+(0,-1)$;)
}

\begin{tikzpicture}

	\coordinate (tr) at (-1,8);
	\coordinate (t1) at (10,6);
	\coordinate (t2) at (7,8);
	\coordinate (b) at (0,3);
	\coordinate (dots) at (3,7);
	
	\ellpath{t1}{b}{$t_1$}{$\ell_1$}{bend right};
	\ellpath{t2}{b}{$t_2$}{$\ell_2$}{bend right};
	\ellpath{tr}{b}{$t_r$}{$\ell_r$}{bend left};
	
	\draw (b) node {$\bullet$} node[below] {$b$};
	
	\draw (dots) node {$\dots$};

	\coordinate (P) at (-1,8);
	\coordinate (L) at (10,6);
	
\end{tikzpicture}}

  }
  \caption{The simple loops $\ell_i$'s around the critical points represent a basis of the homotopy group $\pi_1(\mathbb{C}\setminus \Sigma, b)$. }
  \label{representatives_homotopy}
\end{figure}

We assume that $\infty \in \mathbb{P}^1$ is a regular value (i.e.\ not a critical value) of~$f : Y \to \mathbb{P}^1$ and consider $\mathbb{C}=\mathbb{P}^1 \setminus \{\infty\}$.
A basis of $\pi_1(\mathbb{C}\setminus\Sigma, b)$, the fundamental group of $\mathbb{C}\setminus\Sigma$ pointed at $b$, is given by loops $(\ell_1, \dots,\ell_r)$ around the elements $t_1, \dots, t_r$ of $\Sigma$ (\figref{representatives_homotopy}).
Under the hypothesis that~$f$ is a Lefschetz fibration, the monodromy action on~$H_{n-1}(X_b)$ along these paths are described by the Picard--Lefschetz formula \parencite[\S6.3.3]{Lamotke_1981}:
for $1\le i\le r$,
\begin{equation}
  \label{picard_lefschetz_formula}
  {\ell_i}_*(\eta) = \eta + (-1)^{n(n+1)/2}\langle\eta, \delta_i\rangle \delta_i
\end{equation}
where $\langle\cdot,\cdot\rangle$ is the intersection product on $H_{n-1}(X_b)$ and $\delta_i$ is a cycle in $H_{n-1}(X_b)$, called the {\em vanishing cycle} at $t_i$, determined up to sign.
Note that the vanishing cycle not only depends on $t_i$, but also on the choice of the homotopy class of $\ell_i$.
Note also, as a consequence of~\eqref{picard_lefschetz_formula}, that the map ${\ell_i}_*-\operatorname{id}$ is of rank~1, and its image is generated by $\delta_i$.

\subsubsection{Thimbles}
Decompose the Riemann spere $\Proj^1$ in two closed hemispheres $D_+$ and $D_-$, such that all critical values of $f:Y\to \mathbb{P}^1$ lie in the interior of $D_+$. We pick a base point~$b$ on the equator $S^1 = D_+\cap D_-$.
Finally we denote $Y_+ = f^{-1}(D_+)$, $Y_0 = f^{-1}(S^1)$ and $X_b = f^{-1}(b)$, so that
\begin{equation}
  X_b\subset Y_0\subset Y_+\subset Y.
\end{equation}
For each of the elementary paths~$\ell_i$ (Fig.~\ref{representatives_homotopy}), we consider the map $\tau_{\ell_i} : H_{n-1}(X_b) \to H_n(Y_+, X_b)$
defined in Section~\ref{sec:monodromy-extensions}.
Then~$\tau_{\ell_i}$ has rank one and there is a uniquely determined element~$\Delta_i \in H_n(Y_+, X_b)$,
the \emph{Lefschetz thimble} associated to the critical value~$t_i$,
such that \parencite[(6.7.1)]{Lamotke_1981}
\begin{equation}\label{eq:17}
  \tau_{\ell_i} \eta = -(-1)^{\frac{n(n-1)}{2}} \langle \eta, \delta_i \rangle \Delta_i \in H_{n}(Y_+, X_b),
\end{equation}
which combined with \eqref{relation_extension_monodromy} gives back the Picard--Lefschetz formula~\eqref{picard_lefschetz_formula}.

The thimble~$\Delta_i$ is dependent on the choice of the homotopy class of $\ell_i$.
By definition, it can be obtained as the extension~$\tau_{\ell_i}(p_i)$ of some cycle~$p_i \in H_{n-1}(X_b)$.
It is the generator of~$H_n(f^{-1}(V), X_b)$, for $V$ a simply connected neighbourhood of $\ell_i$ that does not contain any critical value other than $t_i$.
It also can be constructed as the extension of~$\delta_i$ along a path from~$b$ to the singular value~$t_i$ (\figref{fig:thimble}).
In particular, note that
\begin{equation}\label{border_thimble}
  \partial \Delta_i = \delta_i.
\end{equation}

\begin{figure}[tbp]
  \centering
  \begin{subfigure}[b]{0.45\linewidth}
    \begin{center}
      \resizebox{\linewidth}{!}{
  \resizebox{\columnwidth}{!}{\begin{tikzpicture}

	\coordinate (t) at (3,-2);
	\coordinate (x) at (3,1);
	\coordinate (b) at (0,-2);

        	\draw[very thick] (0,0) ellipse (0.25 and 0.5);
	\draw (-0.25, 0) node[left] {$p_i$};
        	\draw[very thick] (0,2) ellipse (0.25 and 0.5);
	\draw (-0.25, 2.5) node[left] {$\ell_{i*}p_i$};
	
	\draw[fill=gray!20] (4,1) edge[in=0, out=-125] (0,0.5) edge[in=0, out=125] (0,1.5);
	\draw (5,1) to[in=0, out=-90] (3,-0.7) to[in=0, out=180] (0,-0.5);
	\draw (5,1) to[in=0, out=90] (3,2.7) to[in=0, out=180] (0,2.5);
	
	\draw (3,2.2) node {$\Delta_i = \tau_{\ell_i}(p_i)$};

	\draw[dashed] (x) node {$\bullet$} node[left] {$x_i$} -- (t) node {$\bullet$} node[left] {$t_i$};
	
	\draw (b) node {$\bullet$} node[left] {$b$};
	\draw (b) --node[midway, below, sloped] {$\longrightarrow$} ($(t)-(0,0.5)$) arc (-90:90:0.5 and 0.5)  --node[midway, above, sloped] {$\longleftarrow$} (b);
	\draw ($(t)+(0.5,0)$) node[right] {$\ell_i$};

\end{tikzpicture}}

      }
    \end{center}
    \caption{There is a unique cycle $\Delta_i \in H_n(Y^{+}, X_b)$ stemming from extensions along the simple loop $\ell_i$ around $t_i$. It is called the {\em Lefschetz thimble at $t_i$}.
    It is dependant on the homotopy class of $\ell_i$.}
  \end{subfigure}
  \hspace{1em}
  \begin{subfigure}[b]{0.45\linewidth}
    \begin{center}
      \resizebox{.8\linewidth}{!}{
  \resizebox{\columnwidth}{!}{\begin{tikzpicture}
	
	\coordinate (x) at (3,0);
	
	\draw[fill=gray!20] (0,1) arc (90:-90:0.5 and 1) -- (x) --  cycle;
	\draw ($(3*2/5,0)$) node {$\Delta_i$};
	
	\draw[very thick] (0,0) ellipse (0.5 and 1);
	\draw (-0.7,0.5)  node[above] {$\delta_i$};
	\draw (x) node {$\bullet$} node[right] {$x_i$};
	
	\draw (0,-2) node {$\bullet$} node[above] {$b$} -- ($(x) + (0,-2)$) node {$\bullet$} node[above] {$t_i$};

\end{tikzpicture}}

      }
    \end{center}
    \caption{The border $\delta_i = {\ell_i}_*p_i - p_i$ of $\Delta_i$ is the {\em vanishing cycle at $t_i$}.
    This name comes from the fact that the geometric representative of $\delta_i$ collapses when deforming towards $t_i$.
    $\Delta_i$ is the extension of $\delta_i$ along a path connecting $b$ to $t_i$.}
  \end{subfigure}
  \caption{Thimbles and vanishing cycles.}
  \label{fig:thimble}
\end{figure}

\subsection{Reconstructing homology}
In order to study the homology of $Y$, we start from the explicit description of the relative homology group~$H_n(Y_+, X_b)$ in terms of the Lefschetz thimbles.

\begin{lemma}[{Main lemma, \cite[\S5]{Lamotke_1981}}]
  \label{lem:main}
  With the notations above,
  \begin{enumerate}[(i)]
    \item $H_q(Y_+, X_b) = 0$ if $q\ne n$;
    \item \label{homology_Y}
          $H_n(Y_+, X_b)$ is free of rank~$r$
          and~$\Delta_1,\dotsc,\Delta_r$ is a basis.
  \end{enumerate}
\end{lemma}

To retrieve~$H_n(Y)$ from the thimbles, the main object is
\begin{equation}
  \mathcal{T}(Y) \eqdef \frac{\ker \left( \partial : H_n(Y_+,X_b) \to H_{n-1}(X_b) \right)}{\im \left( \tau_\infty : H_{n-1}(X_b) \to H_n(Y_+, X_b) \right)},
\end{equation}
where~$\tau_\infty$ is the extension map along the equator.
Let us describe informally the nature of it.
By Lemma~\ref{lem:main}, elements of~$H_n(Y_+, X_b)$ are linear combinations of thimbles.
Thimbles have boundary in~$X_b$ (Fig.~\ref{fig:thimble}), but the boundary of a linear combination of thimbles may be homologically~0 in~$X_b$.
These linear combinations form a subspace which is exactly $\ker \left( \partial : H_n(Y_+,X_b) \to H_{n-1}(X_b) \right)$.
Let~$C$ be such a linear combination. 
By definition, $\partial C$ is homologically $0$ in~$X_b$, so we can add a~$n$-chain~$C'$ in~$X_b$ such that~$\partial(C + C') = 0$. 
Then, the $n$-chain $C+C'$ is a $n$-cycle in~$Y$. Since~$C'$ is determined only up to~$n$-cycles in~$X_b$,
we obtain a well defined element in~$H_n(Y)/\iota_* H_n(X_b)$, where~$\iota : X_b \to Y$ is the inclusion.

Combinations of thimbles produced by the extension map $\tau_\infty : H_{n-1}(X_b) \to H_n(Y_+, X_b)$
are irrelevant, because the equator is contractible in~$\mathbb{P}^1\setminus \Sigma$, so these extensions are homologically zero in~$Y$.
This explains why we have a map
\begin{equation}\label{eq:4}
  \mathcal{T}(Y) \to \frac{H_n(Y)}{\iota_* H_n(X_b)}.
\end{equation}
By composing with~$\pi_* : H_n(Y) \to H_n(X)$, we also obtain a map
\begin{equation}
  \mathcal{T}(Y) \to \frac{H_n(X)}{\iota_* H_n(X_b)}.
\end{equation}

\begin{theorem}\label{lem:TY}
  We have an exact sequence
  \[ 0 \to \mathcal{T}(Y) \to  \frac{H_n(Y)}{\iota_* H_n(X_b)} \to H_{n-2}(X_b) \to 0, \]
  where the arrow from~$\mathcal{T}(Y)$ is~\eqref{eq:4}, and the arrow to~$H_{n-2}(X_b)$
  is given by intersecting with~$X_b$.

  Moreover, let~$K$ be the kernel of the map~$ H_{n-2}(X') \to H_{n-2}(X_b)$ induced by inclusion.
  We have an exact sequence
  \[ 0 \to K \to \mathcal{T}(Y) \to \frac{H_n(X)}{\iota_* H_n(X_b)} \to 0. \]
\end{theorem}

\begin{proof}
  First, we have an isomorphism
  \begin{equation}\label{eq:16}
    H_q(Y, Y_+) \simeq H_{q-2}(X_b),
  \end{equation}
  for any~$q$ \parencite[(3.3.1)]{Lamotke_1981}.
  This is given by excision: $(Y, Y_+) \simeq (Y_-, Y_0)$
  and the fact that~${f:Y_- \to D_-}$ is a trivial fibration since~$D_-$ does not contain critical values.
  The map~${H_q(Y, Y_+) \to H_{q-2}(X_b)}$ is the intersection with~$X_b$.

  Next, consider the commutative diagram
  \begin{equation}\label{diag:bigchase}
    \begin{tikzcd} [sep = .5 cm]
      && H_{n}(X_b)\arrow [d, "\iota_*"]\arrow[rd,"\iota_*"] & & & \\
      0 \arrow [r]& H_{n-2}(X') \arrow[r] \arrow[drr, dashed]& H_n(Y) \arrow [r, "\pi_*"] \arrow [d]& H_n(X) \arrow [r]& 0 &\\
      H_{n-1}(X_b) \arrow [r ,"\tau_\infty"]& H_{n}(Y_+,X_b) \arrow [r] \arrow[dr, "\partial"]& H_n(Y, X_b) \arrow [r] \arrow [d, "\partial"]& H_{n-2}(X_b) \arrow [r]& 0 &\\
      && H_{n-1}(X_b) & & & \,,
    \end{tikzcd}
  \end{equation}
  where:
  \begin{itemize}
    \item The first line comes from \textcite[(3.1.2)]{Lamotke_1981}.
    \item The second line is the long exact sequence of the triple $(Y, Y_+, X_b)$ where $H_q(Y, Y_+)$ is identified to $H_{n-2}(X_b)$ using~\ref{eq:16}. The last term is $0$ because of Lemma~\ref{lem:main}.
    \item The column is the long exact sequence of the pair $(Y, X_b)$.
    \item The dashed arrow is induced by inclusion. Indeed,
          let $A$ be a closed $n-2$-chain in $H_{n-2}(X')$.
          It is sent to $A\times\Proj^1$ in $H_n(Y)$, which is then included in $H_n(Y, X_b)$ and then~$H_n(Y, Y_+) \cap H_{n-2}(X_b)$,
          where the identification is given by intersecting with~$X_b$, which yields $A$ again.
  \end{itemize}
  Importantly, the dashed arrow $H_{n-2}(X') \to H_{n-2}(X_b)$ is surjective, by Lefschetz' hyperplane theorem, since~$X'$ is a hyperplane section of~$X_b$.
  Now, the two exact sequences to prove follow from diagram chasing in~\eqref{diag:bigchase}.
\end{proof}

\subsection{Projective complete intersections}

From now on, we assume that~$X \subseteq \mathbb{P}^N$ is a complete intersection.
In particular, all the homology groups that appear are free, so the exact sequences split, due to the following lemma which gathers some well known facts.

\begin{lemma}\label{torsion_free}
  Let $V$ be a $n$-dimensional smooth complete intersection of $\Proj^{N}$.
  \begin{enumerate}[(i)]
    \item For~$k < n$, the inclusion~$V\hookrightarrow \mathbb{P}^N$ induces isomorphism~$H_k(V)\simeq H_k(\Proj^N)$ and~$H^k(V)\simeq H^k(\Proj^N)$.
    \item All the homology groups~$H_k(V, \Z)$ are free.
    \item For~$k < n$ even, $H_{k}(V)$ is generated by the homology class $\frac{1}{\deg V}[V \cap L]$ where~$L$ is a projective subspace of complex codimension~$n-\frac k2$.
    \item For~$n < k\leq 2n$ even, $H_k(V)$ is  is generated by the homology class $[V \cap L]$ where~$L$ is a projective subspace of complex codimension~$n-\frac k2$.
  \end{enumerate}
\end{lemma}

\begin{proof}
  The first point is a consequence of
  Lefschetz's hyperplane theorem.
  Poincaré duality implies~$H_{k}(V)\simeq H_{2N-2n+k}(\Proj^{N})$ for~$k > n$.
  In particular, $H_k(V)$ is free for any~$k\neq n$.
  For the second point, see \textcite[\S 2.2]{Hirzebruch_1956}.
  For convenience, we recall the main line of the argument.
  By the first point and Poincaré duality, it only remains to check that~$H_n(V)$ is free.
  By the universal coefficient theorem \parencite[Corollary~3.3]{Hatcher_2002}, we have
  \begin{equation}
    H_n(V)\simeq H^n(V) \simeq \operatorname{Free}(H_n(V)) \oplus \operatorname{Tor}(H_{n-1}(V)) = \operatorname{Free}(H_n(V)),
  \end{equation}
  where Free denotes the free part and Tor the torsion part.

  For the third point, consider the inclusion~$V\hookrightarrow \mathbb{P}^N$,
  which, by the first point, induces an isomorphism~$H_k(V) \simeq H_k(\mathbb{P}^N)$.
  It maps a linear section of~$V$ by a projective subspace~$L^{n-\frac k2}$ of codimension~$n-\frac k2$
  to the class of~$V\cap L$ in~$\mathbb{P}^N$. But in~$\mathbb{P}^N$, $V$ is homologous to~$\deg(V) L^{N-n}$
  so~$[V\cap L]$ is homologous to~$\deg(V) L^{N - \frac k2}$.
  So in~$H_k(V)$, $[V \cap L]$ is divisible by~$\deg(V)$ and the quotient is a generator.

  For the last point, we consider the \emph{umkehr} homorphism~$i_!: H_{2N - 2n +k}(\mathbb{P}^N) \to H_{k}(V)$
  obtained by Poincaré duality from the morphism~$H^{2n-k}(\mathbb{P}^N) \to H^{2n-k}(V)$ induced by inclusion.
  As the latter is an isomorphism by the first point, $i_!$ is an isomorphism too.
  We check easily that it maps a projective subspace $L^{n-\frac k2}$, which generates~$H_{2N - 2n +k}(\mathbb{P}^N)$, to the linear section~$V \cap L^{n-\frac k2}$ of~$V$.
\end{proof}

In the case of complete intersections, it is convenient to study the homology of~$X$ without the part coming from the homology of~$\mathbb{P}^N$.
In particular, when~$n$ is even, the homology group~$H_n(X)$ contains the class of a linear section~$h$ of~$X$ by a codimension~$\frac n2$ linear space.
So we define the primitive homology
\begin{equation}
  \label{eq:2}
  \PH_n(X) =
  \begin{cases}
    H_n(X)/\mathbb{Z} h & \text{ if $n$ is even}\\
    H_n(X) &  \text{ if~$n$ is odd}.
  \end{cases}
\end{equation}
In the case where~$n$ is odd, we define~$h = 0 \in H_n(X)$.

Since~$X_b$ is a $n-1$-dimensional complete intersection, the space~$H_n(X_b)$ is generated by a linear section (or zero if~$n$ is odd), which we also denote by~$h$.
In particular, the second exact sequence in Theorem~\ref{lem:TY} now reads
\begin{equation}\label{eq:3}
  0 \to K \to \mathcal{T}(Y) \to \PH_n(X) \to 0.
\end{equation}

We can also be slightly more precise about~$H_n(Y)$ and obtain a decomposition into a direct sum.
The decomposition is not canonical but when we will study the intersection product, natural choices will appear.

\begin{theorem}\label{thm:middle_hom_Y}
  If~$X \subseteq \Proj^N$ is a complete intersection,
  then the middle homology of the modification~$Y$ is given (noncanonically) by
  \begin{equation*}
    H_n(Y) \simeq \mathcal{T}(Y) \oplus H_n(X_b) \oplus H_{n-2}(X_b).
  \end{equation*}
\end{theorem}
\begin{proof}
  Since $H_{n-2}(X_b)$ is free, the first exact sequence in Theorem~\ref{lem:TY} splits.
  Moreover, the map~$\iota : H_{n}(X_b) \to H_n(Y)$ is injective (because after composition with~$\pi_*$, we obtain the inclusion map~$H_n(X_b) \to H_n(X)$ which is injective since it maps a linear section to a linear section).
  As all these groups are free, we obtain the claim.
\end{proof}

\section{Computation of periods}
\label{sec:algorithms}

Let~$n \geq 1$ and let~$X \subseteq \Proj^{n+1}$ be a smooth hypersurface
defined by a homogeneous polynomial~$P \in \C[x_0,\dotsc,x_{n+1}]$ of degree~$d$.
In this section we describe an algorithm for computing the period matrix of~$X$.
When~$n$ is even, the homology group~$H_n(X)$ contains the class of a linear section of~$X$ by a codimension~$\frac n2$ linear space,
which we deal with separately, and similarly in cohomology.

We consider the \emph{primitive De Rham cohomology}, denoted~$\PHdr^n(X)$,
which is a subspace of~$\Hdr^n(X)$ (with coefficients in~$\mathbb{C}$) isomorphic to~$\Hdr^{n+1}(\mathbb{P}^{n+1} \setminus X)$,
and the \emph{primitive homology}, denoted~$\PH_n(X)$, defined in \eqref{eq:2}.
When~$n$ is odd, we have~$\Hdr^n(X) = \PHdr^n(X)$ and~$H_n(X) = \PH_n(X)$.
When~$n$ is even, $\PHdr^n(X)$ is a codimension~1 subspace of~$\Hdr^n(X)$, which is explicitely described by Griffiths--Dwork reduction,
and~$\PH_n(X)$ is the quotient space~$H_n(X)/\mathbb{Z} h$, where~$h$ is the homology class of a linear section of~$X$.
A \emph{primitive period matrix} for~$X$
is the matrix of the pairing~$\PH_n(X) \times \PHdr^n(X) \to \mathbb{C}$, induced by integration,
with a given explicit $\mathbb{C}$-basis of~$\PHdr^n(X)$ and some implicit $\mathbb{Z}$-basis of~$\PH_n(X)$. 
We show this pairing is well defined in Section~\ref{sec:de-rham-middle}.

The algorithm proceeds by induction on the dimension of~$X$.
Let~$H \subseteq \Proj^{n+1}$ be a generic hyperplane.
A primitive period matrix of~$X$ is computed
from a primitive period matrix of~$X \cap H$ (as an hypersurface in~$H \simeq \Proj^n$).
To this end, we choose a pencil of hyperplanes~$\left\{ H_t \right\}_{t\in \Proj^1}$ and a base point~$b$ such that~$H_b = H$
and which induces a \emph{Lefschetz fibration}, as defined in Section~\ref{sec:lefschetz-fibrations}.

Up to a change of coordinates,
we may assume that the pencil is given by $H_t = V(x_{n+1} - t x_0)$.
As in the previous sections,
we consider the rational map~$f : X \dashrightarrow \Proj^1$ given by~$[\mathbf x] \mapsto [x_0:x_{n+1}]$.
We consider the blowup $\pi : Y \to X$ of~$X$ along the base locus~$X' = X\cap V(x_0,x_1)$ of the pencil.
The composition~$f \circ \pi : Y \to \Proj^1$ extends to a regular map on~$Y$, also denoted~$f$.
The fibre~$X_t \overset{\text{def}}{=} f^{-1}(t)$ is isomorphic to the intersection~$X\cap H_t$.
The set of critical values is denoted~$\Sigma$, it is the set of all~$t$ such that~$X_t$ is singular.

\begin{wrapfigure}{r}{.4\linewidth}
  \vspace*{-3ex}
  \begin{center}
    \resizebox{\linewidth}{!}{
  \resizebox{\columnwidth}{!}{\begin{tikzpicture}[-stealth,thick]
 
\node[state] (A) at (0,0){Homology\\ basis};
\node[state] (B) at (4,0){Monodromy\\ matrices};
\node[state] (C) at (2,2.5){Effective\\period matrix};
\draw (A) to[bend left] node[left]{(3)}(C) ;
\draw (C) to[bend left] node[right]{(1)}(B);
\draw (B) to[bend left]node[below]{(2)} (A);
\end{tikzpicture}
 }

    }
  \end{center}
  \vspace*{-6ex}
\end{wrapfigure}

The main steps of the algorithms are:
\begin{enumerate}
  \item derive the action of monodromy of $\pi_1(\mathbb{P}^1\setminus\Sigma)$ on $H_{n-1}(X_b)$ from the primitive period matrix of $X_b$, see Section~\ref{sec:comp-monodr-matr};
  \item compute a basis of homology of $H_n(Y)$ from the action of monodromy, see Section~\ref{sec:comp-periods};
  \item compute a primitive period matrix of~$X$ by integrating the varying periods of~$X_t$ in an appropriate way, see Section~\ref{sec:comp-periods}.
\end{enumerate}
The starting point of this induction is the case of 0-dimensional varieties, where a description of homology and the effective period matrix can be obtained directly.
This case is treated in Section~\ref{sec:dim0init}.
From then on, steps $(1)$, $(2)$, $(3)$ allow to recover the effective period matrix of a curve, another iteration the one of surfaces, and so on.

\subsection{Numerical analytic continuation}
\label{sec:numer-analyt-cont}

To begin with, we briefly present the algorithms underlying our numerical computations.
Consider a linear differential system, in the complex plane, with rational coefficients
-- that is an equation
\begin{equation}
  \label{eq:de}
  Y'(t) = A(t) Y(t),
\end{equation}
where~$A(t) \in \mathbb{C}(t)^{r\times r}$,
and~$Y$ is a unknown vector or matrix of functions.
A point~$t\in\mathbb{C}$ is \emph{ordinary} if~$A$ is continuous at~$t$ and \emph{singular} otherwise.
At any ordinary point~$b\in \mathbb{C}$, there is a uniquely determined $r\times r$ solution matrix~$Y_b$ with~$Y_b(b) = I_r$
\parencite[Theorem~3.1]{Haraoka_2020}.
Let us call it the \emph{fundamental solution at~$b$}.
Any other solution matrix~$\tilde Y$ of~\eqref{eq:de} in a neighbourhood of~$b$ can be written as~$\tilde Y(t) = Y_b(t) U$ for some constant matrix~$U$.

Consider a continuous path~$\gamma : [0,1]\to \mathbb{C}$ which avoids singular values.
From the computational point of view, we assume that~$\gamma$ is polygonal with ordinary vertices in $\mathbb{Q}[i]$ (results in the more general setting where the vertices are singular points or given numerically exist, but we will not need them).
The analytic continuation along~$\gamma$ of the fundamental solution at~$\gamma(0)$
gives a new solution, denoted~$\gamma_* Y$ in the neighbourhood of~$\gamma(1)$ \parencite[Theorem~3.2]{Haraoka_2020}.
In particular, there is a unique matrix~$\Lambda_\gamma \in \mathbb{C}^{r\times r}$, the  \emph{transition matrix along~$\gamma$},
such that
\begin{equation}
  \gamma_* Y = Y_{\gamma(1)} \Lambda_{\gamma}.
\end{equation}
Since~$Y_{\gamma(1)}(\gamma(1)) = \operatorname{id}$, we have $\Lambda_\gamma = \gamma_* Y( \gamma(1) )$.
The map~$\gamma \to \Lambda_\gamma$ is a morphism from the fundamental groupoid of~$\mathbb{C}\setminus\Sigma$
to~$\operatorname{GL}(\mathbb{C}^r)$: $\Lambda_\gamma$ depends on~$\gamma$ only up to homotopy, and if~$\gamma \eta$ is the composition of two paths (going through~$\eta$ first then~$\gamma$),
then~$\Lambda_{\gamma \eta} = \Lambda_{\gamma} \Lambda_\eta$.

\begin{theorem}[\cite{ChudnovskyChudnovsky_1990,VanDerHoeven_1999,Mezzarobba_2010}]
  \label{thm:numcont}
  Given~$p > 0$, we can compute an approximation of~$\Lambda_\gamma$ up to~$2^{-p}$ (for any norm on~$\mathbb{C}^{r\times r}$).
  When the differential system~\eqref{eq:de} and the path~$\gamma$ are fixed, and~$p\to \infty$,
  we need~$O( M(p (\log p)^2)) = p ^{1 +o(1)}$ bit operations, where~$M(n)$ is the complexity of~$n$-bit integer multiplication.
\end{theorem}

We will use this result in two different ways.
First, when~$\gamma$ is a loop, the transition matrix along~$\gamma$
is the monodromy matrix of the differential system~\eqref{eq:de} along~$\gamma$.
Second, we use this result to integrate the solutions of~\eqref{eq:de} along a path.
Let~$R(t) \in \mathbb{C}(t)^{1\times r}$ be a vector of rational functions.
Consider the differential system of dimension~$r+1$
\begin{equation}\label{eq:26}
  Z'= \left(
    \begin{array}{c|c}
      A & 0 \\\hline R & 0
    \end{array}
  \right) Z.
\end{equation}
The vector solutions are exactly of the form~$Z(t) = \left( Y(t), \ g(t) \right)^t$
where~$Y(t)$ is a solution of~\eqref{eq:de} and~$g(t)$ is a primitive integral of~$R(t) \cdot Y(t)$.
More precisely, the fundamental solution at a point~$b$ has the form
\begin{equation}
  Z_b(t) = \left(
    \begin{array}{c|c}
      Y_b(t) & 0 \\\hline
      \int_b^t R(t) \cdot Y_b(t) & 1
    \end{array}
  \right).
\end{equation}
The last row of the transition matrix~$\Lambda_\gamma$ associated to the augmented system~\eqref{eq:26}
is therefore $\int_\gamma R(t) \cdot Y_b(t)$ (and the coefficient~1).
This gives an algorithm for computing integrals over~$\gamma$ of rational function multiples of the coefficients of vector solutions of~\eqref{eq:de}.

As for the practical aspects, we used the implementation of numerical analytic continuation provided in Sagemath by~\textcite{Mezzarobba_2016}
in the \emph{ore\_algebra} package \parencite{KauersJaroschekJohansson_2015}. This package deals with scalar differential equations.
This has practical implications, but from the mathematical point of view, scalar differential equations and differential systems are equivalent \parencite[Chapter~2]{Haraoka_2020}.

\subsection{De Rham middle cohomology classes of a hypersurface}
\label{sec:de-rham-middle}

The understanding of the De~Rham cohomology of $X$ is made simpler by looking at the cohomology of the complement $X^\complement = \Proj^{n+1}\setminus X$ in projective space \parencites[\S2]{Griffiths_1969a}[\S 5.3]{CoxKatz_1999}.

There is a natural injective morphism $\res: \Hdr^{n+1}(X^\complement)\to \Hdr^n(X)$ called the {\em residue mapping}.
Its image is
\begin{equation}
  \PHdr^n(X) \eqdef \res \left( \Hdr^{n+1}(X^\complement) \right) = \left\{ \omega \in \st \int_h \omega = 0 \right\}.
\end{equation}
We also recall the {\em tubular} mapping $T :H_n(X)\to H_{n+1}(X^\complement)$ which sends a cycle to the boundary of a tubular neighbourhood of one of its representative.
These mappings are dual to each other:
for any~$\omega\in H^{n+1}(X^\complement)$ and $\gamma\in H_{n}(X)$,
\begin{equation}\label{eq:11}
  2\pi i \int_\gamma \res\omega = \int_{T(\gamma)}\omega.
\end{equation}
The kernel of~$T$ is generated by~$h$ and we have~$\PH_n(X) = \operatorname{coker}(T) = H_n(X) / \mathbb{Z} h$.
Equation~\eqref{eq:11} shows that the pairing $\PH_n(X) \times \PHdr^n(X)$ is well defined.

The cohomology classes in $H^{n+1}(X^\complement)$ have an explicit description.
Recall that~$P \in \mathbb{C}[x_0,\dotsc,x_{n+1}]$ denotes the defining polynomial of~$X$.
The algebraic differential $n+1$-forms on~$X^\complement$
can be written uniquely as
\begin{equation}
  \label{eq:6}
  \omega = \frac{A}{P^k} \Omega_{n+1},
\end{equation}
where~$\Omega_{n+1} = \sum_{i=0}^{n+1} (-1)^i x_i \ud x_1 \dotsb \ud x_{i-1} \ud x_{i+1} \dotsb \ud x_{n+1}$ is the projective volume form,
$k$ is a positive integer and~$A \in \C[\mathbf x]_{kd-n-2}$ is a homogeneous polynomial of degree~$k d - n - 2$.
Since the variety~$X^\complement$ is affine, its De Rham cohomology can be computed using algebraic forms directly \parencite{Grothendieck_1966}.
Explicitely, we have
\begin{equation}
  \Hdr^{n+1}(X^\complement) \simeq \frac{ \operatorname{Vect}_\C \left\{ \frac{A}{P^k} \Omega_{n+1} \st k\geq 0 \text{ and $A \in \C[\mathbf x]_{kd-n-2}$ }   \right\}}
  { \operatorname{Vect}_\C \left\{ \frac{\partial}{\partial x_i} \left( \frac{B}{P^k} \right) \Omega_{n+1} \st k\geq0\text{ and }0\le i \le n+1 \text{ and } B \in \C[\mathbf x]_{kd-n-1} \right\}}.
\end{equation}
This is the quotient of homogeneous rational functions, regular outside of~$X$, of degree~$-n-2$ modulo sums of partial derivatives
of homogeneous rational functions, regular outside of~$X$, of degree~$-n-1$.

In this representation, a basis of~$\Hdr^{n+1}(X^\complement)$
can be described in terms of the Jacobian ideal of~$f$, namely
\begin{equation}
  \label{eq:19}
  J = \left\langle \frac{\partial P}{\partial x_0}, \dotsc, \frac{\partial P}{\partial x_{n+1}} \right \rangle \subseteq \C[\mathbf x].
\end{equation}
We fix any monomial ordering on~$\C[\mathbf x]$.
A basis of the cohomology space is given by the forms~$\frac{m}{P^k} \Omega_{n+1}$,
where~$m$ is a monomial in the variables~$\mathbf x$ such that:
\begin{itemize}
  \item $k d = \deg(m) + n+ 2$;
  \item $m$ is not the leading monomial of any element of~$J$.
\end{itemize}
The process of computing the coefficients in this basis of a given equivalence class of a $n+1$-form on~$X^\complement$
is called \emph{Griffiths--Dwork reduction}. For more details, we refer to \textcite[\S 5.3]{CoxKatz_1999}.

\subsection{Gauss--Manin connection}
\label{sec:gauss-manin-conn}

Let~$P_t = P(x_0,\dotsc,x_n, t x_0)$, where~$P \in \mathbb{C}[x_0,\dotsc,x_{n+1}]$ is the defining polynomial of~$X$,
so that~$X_t$ identifies with~$V(P_t) \subset \mathbb{P}^n$.
Following Section~\ref{sec:de-rham-middle}, the residue mapping induces an isomorphism
\begin{equation}
  \PHdr^{n-1}(X_t) \simeq \frac{ \operatorname{Vect}_\C \left\{ \frac{A}{P_t^k} \Omega_{n} \st k\geq 0 \text{ and $A \in \C[ x_0,\dotsc,x_n]_{kd-n-1}$ }   \right\}}
  { \operatorname{Vect}_\C \left\{ \frac{\partial}{\partial x_i} \left( \frac{B}{P_t^k} \right) \Omega_{n} \st k\geq0\text{ and }0\le i \le n \text{ and } B \in \C[x_0,\dotsc,x_n]_{kd-n} \right\}}.
\end{equation}
By extending the scalars from~$\mathbb{C}$ to~$\mathbb{C}(t)$,
we can define the \emph{space of sections} of~$\PHdr^{n-1}(X_t)$
as the~$\mathbb{C}(t)$-vector space
\begin{equation}
   \mathcal{H} =  \frac{ \operatorname{Vect}_{\C(t)} \left\{ \frac{A}{P_t^k} \Omega_{n} \st k\geq 0 \text{ and $A \in \C(t)[ x_0,\dotsc,x_n]_{kd-n-1}$ }   \right\}}
  { \operatorname{Vect}_{\C(t)} \left\{ \frac{\partial}{\partial x_i} \left( \frac{B}{P_t^k} \right) \Omega_{n} \st k\geq0\text{ and }0\le i \le n \text{ and } B \in \C(t)[x_0,\dotsc,x_n]_{kd-n} \right\}}.
\end{equation}
For a given~$\beta \in \mathcal{H}$ and a generic~$t \in \mathbb{P}^1$, evaluation at~$t$ gives an element~$\beta(t)$ of~$\PHdr^n(X_t)$.
Because~$\frac{\partial}{\partial t}$ commutes with~$\frac{\partial}{\partial x_i}$, we see that differentiation with respect to the parameter~$t$ induces
a derivation~$\nabla$ of~$\mathcal{H}$,
called the \emph{Gauss--Manin connection}.

Let~$\beta \in \mathcal{H}$.
For~$\eta\in H_{n-1}(X_b)$, let~$\eta(t) \in H_{n-1}(X_t)$ be the cycle uniquely determined by transporting~$\eta$ in~$Y_t$, following a path in some simply connected neighbourhood of~$b$.
The Gauss--Manin connection is uniquely determined by the property
\begin{equation}\label{eq:14}
  \frac{d}{dt} \int_{\eta(t)} \beta(t) = \int_{\eta(t)} \nabla\beta(t).
\end{equation}

Using Griffiths--Dwork reduction, we can compute a basis of~$\mathcal{H}$, say~$\beta_1,\dotsc,\beta_s$,
and the matrix~$A(t) = (a_{ij}) \in \mathbb{C}(t)^{r\times r}$ of the Gaussin--Manin connection (for details on its computation, see~\textcite{BostanLairezSalvy_2013}), defined by
\begin{equation}\label{eq:15}
  \nabla \beta_i = \sum_{j} a_{ij} \beta_j.
\end{equation}
Typically, $\beta_i$ will be in the form~$\frac{m}{P_t^k} \Omega_n$,
for some monomial~$m$,
and so~$\nabla \beta_i$ will be given by the Griffiths--Dwork reduction of~$-k \frac{m}{P_t^{k+1}} \frac{\partial P_t}{\partial t} \Omega_n$.
We can also choose~$\beta_1$ to be a cyclic vector and~$\beta_i = \nabla^{i-1} \beta_1$ so that the differential system is encoded in a single scalar equation which is well adapted to the \emph{ore\_algebra} package (see Section~\ref{sec:numer-analyt-cont}).
We may assume that the evaluations~$\beta_i(b)$ yield a basis of~$\PHdr^n(X_b)$, as this will generically be true.

\subsection{Primitive period matrix of a 0-dimensional hypersurface}
\label{sec:dim0init}
We start with the simple case where~$X$ is a zero-dimensional variety in~$\mathbb{P}^1$.
The periods are directly given by residues of rational fractions.

More precisely, the middle homology group~$H_0(X)$ is freely generated by~$d$ points.
The cohomology space~$\Hdr^0(X)$ is the set of functions~$X\to \mathbb{C}$.
The linear section~$h$ is the sum of all points.
In particular~$\PHdr^0(X)$ is the set of functions~$r : X\to \mathbb{C}$ with~$\sum_{x\in X} r(x) = 0$.
A basis of the middle primitive cohomology space $\Hdr^1(X^\complement)$ is given by the rational forms
$\omega_k = \frac{x^ky^{d-k-2}}{P}\Omega_1$ for $0\le k\le d-2$.
The residue mapping~$\Hdr^1(X^\complement) \to H^0(X)$ is the classical residue
\begin{equation}
  \frac{A}{P} \Omega_1 \mapsto \left( z \in X \mapsto \operatorname{Res}_{z} \left( \frac{A}{P} \Omega_1 \right) \right),
\end{equation}
the tube map~$T : H_0(X) \to H_1(X^\complement)$ maps a point of~$X$ to a loop around it,
and Equation~\eqref{eq:11} is Cauchy's residue theorem.
Since, the sum of residues is zero, the image of the residue mapping is indeed included in~$\PHdr^0(X)$.

We choose~$d-1$ roots~$z_1,\dotsc,z_{d-1}$ of $P$ in~$X$ in the affine chart~$y=1$,
they give a basis of~$\PH_0(X)$ and a primitive period matrix of~$X$ is given by the coefficients
\begin{equation}
  z_i^j \frac{\partial P}{\partial x}(z_i)^{-1},
\end{equation}
for $1\leq i \leq d-1$ and~$0\leq j \leq d-2$.

\subsection{Computation of the monodromy matrices of a Lefschetz fibration}
\label{sec:comp-monodr-matr}

\subsubsection{Monodromy on the primitive homology}

We consider the monodromy action of~$\pi_1(\mathbb{P}^1 \setminus \Sigma)$ on~$H_{n-1}(X_b)$.
Let~$\ell$ be a loop~$\mathbb{P}^1 \setminus \Sigma$, starting from the base point~$b$.
As~$t$ runs along~$\ell$, $X_t$ is continously deformed and
there is a uniquely determined continuation~$\eta(t)$ of~$\eta$ in~$H_{n-1}(X_t)$.
The action of~$\ell$ on~$\eta$, denoted~$\ell_* \eta$ is defined as the determination of~$\eta(b)$
after~$t$ has travelled along~$\ell$.
It is clear that a linear section of~$X_b$ has trivial monodromy,
thus the monodromy action on $H_{n-1}(X_b)$ induces an action on~$\PH_{n-1}(X_b)$.

\subsubsection{Computation of a monodromy matrix given a path}
\label{sec:comp-monodr-matr-1}

Let~$\Pi(t)$ be a primitive period matrix of~$X_t$ defined by
\begin{equation}
  \label{eq:13}
  \Pi_{ij}(t) = \int_{\eta_{j}(t)} \beta_i(t),
\end{equation}
for some basis~$\eta_j$ of~$\PH_{n-1}(X_b)$,
and where the~$\beta_i(t)$'s form a basis of~$\mathcal{H}$, see Section~\ref{sec:gauss-manin-conn}.
It depends holomorphically on~$t$ (on some neighbourhood of~$b$).
Combining~\eqref{eq:14} and~\eqref{eq:15}, we obtain the first-order linear differential system
\begin{equation}\label{eq:1}
  \Pi'(t) = A(t) \Pi(t).
\end{equation}
In particular, the matrix~$\Pi$ extends holomorphically along the path~$\ell$
and gives another determination, denoted~$\ell_* \Pi$, of~$\Pi$ in a neighbourhood of~$b$.
Naturally,
\begin{equation}
  \ell_* \Pi_{ij}(t) = \int_{\ell_* \eta_{j}(t)} \beta_i(t).
\end{equation}
In particular, the matrix of the action of~$\ell$ on~$H_{n-1}(X_b)$ in the basis~$(\eta_j)$ is given by
\begin{equation}\label{eq:27}
  \operatorname{Mat}(\ell_*) = \Pi^{-1}(b) \cdot \ell_* \Pi(b) = \Pi^{-1}(b) \Lambda_\ell \Pi(b),
\end{equation}
where~$\Lambda_\ell$ is the transition matrix introduced in Section~\ref{sec:numer-analyt-cont},
associated to the differential system~\eqref{eq:1} and the loop~$\ell$.
It is possible to compute~$\Delta_\ell$ numerically with arbitrary precision and rigorous error bounds
using the differential system~\eqref{eq:1}, see Theorem~\ref{thm:numcont}.
Together with the data of~$\Pi(b)$, we compute~$\operatorname{Mat}(\ell_*)$.

\subsubsection{Computation of appropriate paths}
\label{sec:comp-appr-paths}

It only remains to compute a set of generators of $\pi_1(\mathbb{P}^1 \setminus \Sigma)$,
which we compute as piecewise linear paths.
To do this, we compute the Voronoi diagram of $\Sigma$ in~$\mathbb{C}$.
Each critical point $c_i\in \Sigma$ lies in a unique Voronoi cell. 
Up to adding additional points around the convex hull of $\Sigma$, we may assume that the boundary of that cell is a polygon that describes a loop~$\ell'_i$ around the critical point (although not pointed at $b$). 
We pick the loop to be anticlockwise.

For each $i$, we pick a vertex $v_i$ of $\ell_i'$.
We can then consider the Voronoi graph $V$, for which the vertices and edges are those of the Voronoi cells. 
We also add the basepoint $b$, and an edge connecting the basepoint to the closest other vertex in $V$.
We compute a subtree $T$ of $V$ covering all the $v_i$ and rooted at $b$. 
In $T$ there is a unique path $p_i$ connecting $b$ to a given $v_i$.
A simple loop around $c_i$ pointed at $b$ is then given by the composition $\ell_i = p_i^{-1} \ell'_i p_i$.

For the sake of computing the extension around the equator $\tau_\infty$ (see \eqref{eq:infinity_cycles} below), we need to order these paths so that the composition of them is the loop around $\infty$. 
We define a supertree $T'$ of $T$ by adding a child corresponding to $\ell_i'$ at $v_i$, for each $i$. 
For a given node of $T'$, we order its children in anticlockwise order starting from the parent.
Finally the ordering on the loops is simply the ordering induced by the prefix ordering of the nodes of $T'$.
This can be achieved with a depth-first search throughout $T$, illustrated in \figref{fig:depth_sort}.

\begin{figure}[ht]
    \centering
    \resizebox{0.8\linewidth}{!}{
  \resizebox{\columnwidth}{!}{\begin{tikzpicture}
	\coordinate (C) at (-0.70 ,  -3.10);
        \coordinate (D) at (2.880 ,  -1.870);
        \coordinate (E) at (3.60 ,  -0.560);
        \coordinate (F) at (6.40 ,  -0.60);
        \coordinate (G) at (7.30 ,  1.80);
        \coordinate (H) at (8.30 ,  -1.70);
        \coordinate (I) at (8.30 ,  -2.850);
        \coordinate (J) at (12.640 ,  1.80);
        \coordinate (K) at (12.430 ,  -0.70);
        \coordinate (L) at (4.370 ,  -2.510);
        \coordinate (M) at (5.20 ,  -2.410);
        \coordinate (N) at (3.870 ,  4.280);
        \coordinate (O) at (2 ,  0);
        \coordinate (P) at (13.70 ,  2.70);
        \coordinate (Q) at (13.10 ,  -1.40);
        \coordinate (R) at (-7.50 ,  -4.70);
        \coordinate (S) at (6.3 ,  -3.48);
        \coordinate (T) at (0.75 , 1.9);
        
        \coordinate (A) at (9.6 , 0.36);
        \coordinate (B) at (4.60 ,  -1.430);
        
        \coordinate (A2) at (9.22 ,  0.2);
        \coordinate (B2) at (4.22 ,  -1.57);
	
	\draw[ultra thick] (C) -- node[midway, above]{$T$} (D) -- (E) -- (F) -- (G) -- (N);
	\draw[ultra thick] (F) -- (H) -- (I);
	
	\draw[dashed] (D) -- (L) -- (M) -- (F);
	\draw[dashed] (N) -- (T) -- (O) -- (E);
	\draw[dashed] (H) -- (K) -- (Q);
	\draw[dashed] (K) -- (J) -- (G);
	\draw[dashed] (J) -- (P);
	\draw[dashed] (S) -- (M);
	
	\draw (A) node{$\ell'_i$};
	\draw (B) node{$\ell'_j$};
	
	\draw [->] (B2) arc(-160: 160: 0.4);
	\draw [->] (A2) arc(-160: 160: 0.4);
	
	\draw (E) node[above]{$p$};
	\draw (F) node[above left]{$v$};
	\draw (H) node[left]{$n_1$};
	\draw (G) node[above]{$n_2$};
	
\end{tikzpicture}}

        }
       	\caption{The tree $T$ (bold) is a subtree of the Voronoi graph (dashed) that connects all the points on which the polygonal loops are pointed. 
	Assume we are visiting vertex $v$, coming from vertex $p$, and that two polygonal loops $\ell'_i$ and $\ell'_j$ are pointed at $v$. 
	$v$ has three neighbours $p$, $n_1$ and $n_2$. 
	The sorting algorithm will yield $j$, the result of visiting $n_1$, $i$, and finally the result of visiting $n_2$, in this order. 
	Applying this sorting algorithm from the vertex $b$ gives an order on the loops pointed at $b$ such that their composition is the loop around $\infty$.
	}
	\label{fig:depth_sort}
\end{figure}

\subsection{Computation of a homology basis}\label{sec:comp-hom-basis}

We seek a description of~$\PH_n(X)$ given the matrices of the monodromy action of~$\pi_1(\mathbb{P}^1\setminus \Sigma)$ on~$\PH_{n-1}(X_b)$.
Firstly, $\PH_n(X)$ identifies with a quotient of~$\mathcal{T}(Y)$, by~\eqref{eq:3}.
By computing~$\PH_n(X)$, we mean, first, computing~$\mathcal{T}(Y)$ and, second, computing the kernel of~$\pi_* : \mathcal{T}(Y) \to \PH_n(X)$.

\subsubsection{Basis of $\mathcal{T}(Y)$}

Recall the setting of Section~\ref{sec:monodromy-extensions}:
we have a generating set~$\ell_1,\dotsc,\ell_r$ of~$\pi_1(\mathbb{P}^1\setminus \Sigma, b)$
and we further assume that~$\ell_r \dotsb \ell_1 = 1$ is the loop at $\infty$.
Each~$\ell_i$ induces an automorphism of~$\PH_{n-1}(X_b)$, denoted~${\ell_i}_*$.
Thanks to the algorithm given in Section~\ref{sec:comp-monodr-matr},
we can compute the matrix~$M_i$ of~${\ell_i}_*$ with respect to some fixed basis~$\beta_1,\dotsc,\beta_s$ of~$\PH_{n-1}(X_b)$.

The linear section~$h$ (which is nonzero when~$n$ is even) is fixed by~${\ell_i}_*$
By the Picard--Lefschetz formula \eqref{picard_lefschetz_formula},
this implies that~$\langle h,\delta_i \rangle = 0$. 
In particular, $\langle h,h\rangle = \deg X\ne0$, and therefore $h$ is not a vanishing cycle.
These two observations show that monodromy action on the quotient is well defined, 
and the corresponding matrices still satisfy~$\operatorname{rk} M_i - I_s=1$.
Thus there are vectors~$d_i \in \mathbb{Z}^{s\times 1}$ and~$m_i \in \mathbb{Z}^{1\times s}$ such that
\begin{equation}
\label{eq:decomp_mono_mat}
  M_i = I_s + d_i  m_i \in \mathbb{Z}^{s\times s}.
\end{equation}
The vectors~$d_i$ and~$m_i$ are uniquely determined, up to sign, by~$M_i$.
The vector~$d_i$ is the coordinates of the vanishing cycle~$\delta_i$ (see~\S\ref{sec:monodromy-extensions}) in the basis~$\beta_1,\dotsc,\beta_s$ of~$\PH_{n-1}(X_b)$,
and~$m_i$ is the linear form~$\eta \mapsto (-1)^{\frac{n(n+1)}{2}} \langle \eta, \delta_i\rangle$, which is well defined on~$\PH_{n-1}(X_b)$ since~$\langle h,\delta_i \rangle = 0$.

The extension map~$\tau_{\ell_i} : H_{n-1}(X_b)\to H_n(Y_+, X_b)$ is described by~\eqref{eq:17} in terms of the linear maps~$\langle -, \delta_i \rangle$.
In particular, it factors through~$\PH_{n-1}(X)$, and
in the bases~$(\beta_j)_{1\leq j \leq s}$ for~$\PH_{n-1}(X_b)$ and~$(\Delta_i)_{1\leq i\leq r}$ for~$H_n(Y_+, X_b)$,
the matrix~$T_i$ of the induced map is given by
\begin{equation}
\label{eq:extension_matrix}
  T_i = 
  (-1)^{n-1}   \left(\begin{array}{ccc} &\mathbf 0& \\\hline &m_i&\\\hline &\mathbf 0& \end{array}\right)  \in \Z^{r\times s}\,,\text{ where $m_i$ is the $i$-th line.}
\end{equation}
Finally, the boundary map~$\delta : H_n(Y_+, X_b) \to H_{n-1}(X_b)$ mapping~$\Delta_i$ to~$\delta_i$
induces a map ${\tilde \delta : H_n(Y_+, X_b) \to \PH_{n-1}(X_b)}$.
The matrix~$B$ of~$\tilde \delta$
is given by
\begin{equation}
  \label{eq:9}
  B = \operatorname{Mat}(\tilde \delta) = \left(
    \begin{array}{c|c|c}
      & & \\
      d_1 & \dotsb & d_r \\
      & &
    \end{array}
  \right) \in \mathbb{Z}^{s\times r}.
\end{equation}
Recall that~$\tau_\infty$ is the extension map~$H_{n-1}(X_b) \to H_n(Y_+, X_b)$
along the equator, which is homotopically equivalent to the composition~$\ell_r \dotsb \ell_1$.
By~\eqref{extension_composition}, it follows that
\begin{equation}
  \label{eq:10}
  \tau_\infty = \sum_{i=1}^r \tau_{\ell_i} \ell_{i-1*} \dotsb \ell_{1*},
\end{equation}
and therefore, in terms of the notations above, the matrix of the map~$\PH_{n-1}(X_b) \to H_n(Y_+, X_b)$ induced by~$\tau_\infty$ is
\begin{equation}
  \label{eq:infinity_cycles}
  T_\infty = T_1 + T_2 M_1 + T_3 M_2 M_1 + \dotsb + T_r M_{r-1} \dotsb M_1.
\end{equation}
With these matrices in hands, we obtain a basis of~$\mathcal{T}(Y)$ as~$\ker(B) / \operatorname{im}(T_\infty)$.

\subsubsection{Kernel of $\mathcal{T}(Y) \to \PH_n(X)$}
\label{sec:kernel-TY-to-PHX}

The final step to get $\PH_n(X)$ is to identify the cycles stemming from the blowup of the base locus of the fibration $Y\to X$.
We propose two methods.
The first one, which we implemented, uses the duality between~$\PH_n(X)$ and~$\PHdr^n(X)$ to obtain that
\begin{equation}
  \label{eq:20}
  \ker \left( \pi_* : \mathcal{T}(Y) \to \PH_n(X) \right) = \left\{ \gamma \in \mathcal{T}(Y) \st \forall \omega \in \PHdr^n(X), \int_{\pi_*(\gamma)} \omega = 0 \right\}.
\end{equation}
This amounts to computing the kernel of a full-rank matrix with complex coefficients, knowing that it is generated by integer-coefficient vectors.
This is numerically stable as the matrix we consider has full rank.
In practice, the coefficients are small and we can compute them using lattice-reduction algorithms.
However, we cannot certify this computation. 
In Section~\ref{sec:removing_exc_divs} we give a certified method for finding the coefficients of the blowups in terms of the thimbles in the case of surfaces.
This method generalises to higher dimensions but we did not implement it.

Since the matrix of the pairing~$\PH_n(X) \times \PHdr^n(X)$ is non-degenerate, the kernel of~$\pi_*$ is
exactly the left-kernel of the full-rank matrix~$P_{Y, X}$.
This is a sublattice of~$\mathcal{T}(Y)$, so~$\ker \pi_*$ is generated by integer-coefficient vectors.
We present an alternative rigorous way of computing~$\ker \pi_*$ in Section~\ref{sec:removing_exc_divs}.

\subsection{Computation of the effective period matrix}\label{sec:comp-periods}
We now have a basis of~$\mathcal{T}(Y)$
and, by~\eqref{eq:3}, $\pi : Y\to X$ induces a surjective map~$\mathcal{T}(Y) \to \PH_n(X)$.
In order to compute the period matrix of~$H_n(X)$, we first compute the matrix~$P_{Y,X}$ of the pairing $\mathcal{T}(Y) \times \PHdr^n(X) \to \C$
\begin{equation} 
	(\gamma,  \omega) \mapsto \int_{\pi_*(\gamma)} \omega.
\end{equation}
To recover a primitive period matrix of~$X$, it is then sufficient to compute a basis of~$\mathcal{T}(Y)/\ker \pi_*$
and extract from~$P_{Y,X}$ the relevant submatrix.

\subsubsection{Integrating periods}
\label{sec:thimble_integration}

Let $\omega\in \PHdr^n(X)$
and~$\gamma \in \mathcal{T}(Y)$.
By definition, $\gamma$ is the extension along some path~$\ell$ in~$\mathbb{P}^1\setminus \Sigma$
of some cycle~$\eta \in H_{n-1}(X_b)$.
Assume for now that we can write~$\pi^* \omega = \beta \wedge \ud f$, for some~$n-1$-form~$\beta$ on~$Y$. 
Then
\begin{equation}
  \int_{\pi_*\gamma} \omega = \int_{\gamma} \pi^*\omega = \int_{\tau_\ell(\eta)}\beta\wedge \ud f = \int_{\ell} \left(\int_{\eta_t} \beta|_{X_t} \right)\ud t,
\end{equation}
where~$\eta_t$ is the uniquely determined continuation of~$\eta \in H_{n-1}(X_t)$.
This expresses the period~$\int_{\pi_* \gamma} \omega$ as an integral of a period of the fiber~$X_t$, varying with t

The computations can be more explicitly carried out using the isomorphism between $\PHdr^n(X)$ and $\Hdr^{n+1}(X^\complement)$.
Recall that~$X^\complement$ denotes the complement~$\mathbb{P}^{n+1}\setminus X$.
Let~$Y'= Y\setminus Y_\infty$, where~$Y_\infty$ is the fibre of~$f : Y \to \mathbb{P}^1$ above the point at infinity.
By choice of coordinates, the hyperplane family~$(H_t)_{t\in \mathbb{P}^1}$
is given by~$H_t = V(x_{n+1} - t x_0)$, and
we check that
\[ Y' \simeq \left\{ (x,t) \in \mathbb{P}^n\times\C \st P_t(x_0,\dotsc,x_n) = 0 \right\}, \]
where~$P_t(x_0,\dotsc,x_n) = P(x_0,\dotsc,x_n, tx_0)$ is the equation of~$X_t$ in~$\mathbb{P}^n$.
Let~$Y^\complement$ be the complement of~$Y'$, that is
\[ Y^\complement = \left\{ (x, t) \in \mathbb{P}^n\times\C \st P_t(x_0,\dotsc,x_n) \neq 0 \right\}. \]
The map
\[ ([x_0:\dotsb:x_n], t) \mapsto \left[ x_0 : \dotsb : x_n : t x_0 \right] \]
induces a map~$\pi : Y^\complement \to X^\complement$.
The Leray residue maps~$H^{n+1}(Y^\complement) \to H_{n}(Y')$ and~$H^{n+1}(X^\complement) \to H^n(X)$
commute with~$\pi$.
Therefore, given a form~$\omega \in \PHdr^n(X)$,
which we write as~$\res( \frac{A}{P^k} \Omega_{n+1} )$ following Section~\ref{sec:de-rham-middle},
we have
\begin{align}
  \label{eq:8}
  \int_{\pi_*\gamma} \omega & = \int_{\gamma} \pi^* \res\left(  \frac{A}{P^k} \Omega_{n+1} \right)
                            = \int_{T(\gamma)} \pi^*\left(\frac{A}{P^k} \Omega_{n+1}\right)
                              = \int_{T(\gamma)} \frac{x_0 A_t}{P_t^k} \Omega_{n} \wedge \ud t \\
                            &=\int_{\ell} \left( \int_{T(\eta_t)} \frac{x_0 A_t}{P^k_t} \Omega_{n} \right) \ud t.
\end{align}
The form~$\frac{x_0 A_t}{P^k_t} \Omega_{n}$
defines an element of the space~$\mathcal{H}$ of sections of~$\PHdr^n(X_t^\complement)$.
In particular, we can write, in~$\mathcal{H}$,
\[ \frac{x_0 A_t}{P^k_t} \Omega_{n} = \sum_{i=1}^s r_i(t) \beta_i(t) \]
for some rational functions~$r_i(t)$, which we can compute explicitely using the Griffiths--Dwork reduction.
The vector~$Y'= (y_i(t))_{1\leq i\leq s}$ defined by~$y_i(t) = \int_{\eta_t} \beta_i(t)$
is a solution to the differential system~$Y'(t) = A(t) Y(t)$ coming from the Gauss--Manin connection, see~\eqref{eq:15} and~\eqref{eq:1}.
Moreover, \eqref{eq:8} gives
\begin{equation}\label{eq:5}
    \int_{\pi_*\gamma} \omega = \int_{\ell} \sum_{i=1}^s r_i(t) y_i(t) \ud t.
\end{equation}
We can compute~$Y(0)$ from the primitive period matrix of~$X_b$ (which we assume is given as input data),
and then we can use numerical analytic continuation to compute the integral in~\eqref{eq:5} efficiently (Section~\ref{sec:numer-analyt-cont}).

\subsection{Wrapup}\label{sec:wrapup}

Let us summarise the main steps of the algorithm for computing a primitive period matrix of
a projective hypersurface~$X$ in~$\mathbb{P}^{n+1}$
given:
\begin{itemize}
  \item the defining polynomial~$P(x_0,\dotsc,x_{n+1})$ of~$X$;
  \item a generic hyperplane family~$H_t = V(x_{n+1} - tx_0)$;
  \item a generic base point~$b$;
  \item a primitive period matrix~$\Pi(b)$ of~$X_b = X \cap H_b$, for a well-specified basis of~$\PHdr^n(X_b)$.
\end{itemize}

\begin{enumerate}
  \item\label{item:1} Using Griffiths--Dwork reduction, compute a basis~$\beta_1(t),\dotsc,\beta_s(t)$ of~$\mathcal{H}$ (as a $\mathbb{C}(t)$-linear space),
  the space of sections of~$\PHdr^{n-1}(X_t)$ defined in Section~\ref{sec:gauss-manin-conn},
  and the matrix~$A(t) \in \mathbb{C}(t)^{s\times s}$ of the Gauss--Manin connection over it (\S\S \ref{sec:de-rham-middle} and~\ref{sec:gauss-manin-conn}).

  \item If necessary, perform a change of basis so that the primitive period
  matrix~$\Pi(b)$ of~$X_b$ is given with respect to the
  basis~$(\beta_i(b))$ of~$\PHdr^n(X_b)$.

  \item \label{item:2}
  Compute the critical values~$\Sigma \subset \mathbb{P}^1$ and polygonal
  loops~$\ell_1,\dotsc,\ell_r$ generating the fundamental group of~$\mathbb{P}^1 \setminus \Sigma$ (\S \ref{sec:comp-appr-paths}).

  \item For each~$\ell_i$, integrate numerically the differential system~$\Pi'(t) = A(t) \Pi(t)$, with given initial value~$\Pi(b)$  along~$\ell_i$
  to obtain the value~${\ell_i}_* \Pi(b)$, with rigorous error bounds.
  Compute the matrix~$M_i = \Pi(b)^{-1} \cdot {\ell_i}_* \Pi(b)$. It is an integer matrix, so we only need to compute the coefficients of~$M_i$ with a coefficient-wise error bounded by~$\frac 12$ (\S \ref{sec:comp-monodr-matr-1}).

  \item\label{it:tubes} Using Equations~\eqref{eq:9} and~\eqref{eq:infinity_cycles}, compute the integer matrices~$B \in \mathbb{Z}^{s\times r}$ and~$T_\infty \in \mathbb{Z}^{r\times s}$. Compute bases of~$\ker(B)$ and $\operatorname{im}(T_\infty)$ in~$\bigoplus_{i=1}^r \mathbb{Z} \Delta_i$
  and a basis of a sublattice~$T \subseteq \ker(B)$ such that~$\ker(B) = T \oplus \operatorname{im}(T_\infty)$.
  This sublattice is isomorphic to~$\mathcal{T}(Y)$ (\S \ref{sec:comp-hom-basis}).

  \item\label{item:4} Compute a basis~$\omega_1,\dotsc,\omega_e$ of~$\PHdr^n(X)$, using Griffiths--Dwork reduction,
  and compute the integrals~$\int_{\pi^* (\Delta_i)} \omega_j$ (\S \ref{sec:thimble_integration}). This amounts to

  \begin{enumerate}
    \item Computing the coefficients of~$\omega_j|_{X_t}$ in the basis~$(\beta_i)$ of~$\mathcal{H}$;
    \item Computing a Picard--Fuchs differential equation for the partial integral~$y(t) = \int_{\delta_i(t)} \omega_j|_{X_t}$,
    using the coefficients above and the matrix~$A(t)$ of the Gauss--Manin connection, where~$\delta_i(t)$ is the vanishing cycle associated to the $i$-th critical value, transported in~$X_t$;
    \item Computing initial conditions~$y^{(k)}(b)$ using the matrix~$\Pi(b)$;
    \item\label{item:3} Computing~$\int_{\ell_i} y(t) \ud t$ using the method given in Section~\ref{sec:numer-analyt-cont}.
  \end{enumerate}

  \item With appropriate linear combinations, compute~$\int_{\pi^*(\tau_i)} \omega_j$ for some basis~$(\tau_i)$ of~$T \simeq \mathcal{T}(Y)$ computed at Step~\ref{it:tubes},
  which gives the matrix~$P$ of the pairing~$\mathcal{T}(Y) \times \PHdr^n(X)$.

  \item Numerically compute the left kernel of this matrix, this is a subgroup of~$\mathcal{T}(Y)$.
  Restrict this matrix to a supplement of this kernel, to obtain a primitive period matrix of~$X$.
\end{enumerate}

\subsubsection*{Complexity aspects}
Let~$d$ be the degree of~$P$.
We can perform Step~\ref{item:1} using~$d^{5n + O(1)}$ operations \parencite{BostanLairezSalvy_2013}.
The dimension~$s$ of~$\mathcal{H}$ is bounded by~$d^n$
and the entries of the~$s\times s$ matrix~$A$ are rational functions with numerators and denominators of degree at most~$n 3^n d^{n+1}$
\parencite[Proposition~8]{BostanLairezSalvy_2013}.

In Step~\ref{item:2}, the set~$\Sigma$ of critical values has~$d(d-1)^n$ elements (by genericity of the hyperplane family).
We can compute them by solving a system of~$n+1$ homogeneous equations of degree at most~$d$ in~$\mathbb{P}^{n+1}$.
The algebraic complexity is bounded by~$d^{O(n)}$ \parencite{GiustiLecerfSalvy_2001},
and we should also consider the cost of numerical approximation in the complex plane.
In practice, the computational cost is negligible. The polynomial systems we have can be considered as toy examples.

In Step~\ref{it:tubes}, we reduce to integrating a differential operator of order at most~$d^n$ and coefficients of degree~$d^{O(n)}$.
As an optimization, Steps~\ref{it:tubes} and~\ref{item:3} can be performed simultenously.
The complexity depends on the size of the differential operator, the desired precision
but also numerical parameters. No complete description is known. With respect to precision only, when everything else is fixed,
Theorem~\ref{thm:numcont} guarantees a quasilinear complexity.

\section{An explicit example: quartic surface}
\label{example}

Let $P = w^4 + x^4 + y^4+ z^4$ and define the Fermat quartic surface $X = V(P) \subset\Proj^{3}$. 
It is a smooth quartic projective surface. 
Thus it is a $K3$ surface, its middle cohomology group has rank~22 and its holomorphic subgroup has rank~1.
In this section, we give an explicit description of the computation of the periods of $X$.

A static SAGE worksheet reproducing the computations of this section can be found at \mbox{\emph{Fermat\_periods.ipynb}}\footnote{\url{https://nbviewer.org/urls/gitlab.inria.fr/epichonp/eplt-support/-/raw/main/Fermat_periods.ipynb}}. 
The computation of this notebook took a bit less than 18 minutes on a laptop.

\subsection{Constructing the Lefschetz fibration}
Let $\lambda = w$ and $\mu = 2x+3y+z$, and for $t\in \Proj^1$, define $H_t = V(\lambda - t\mu)$.
This defines a hyperplane pencil $\{H_t\}_{t\in \Proj^1}$ with axis $A = V(\lambda, \mu)$.
Then the modification of $Y$ along $X$ is the blowup of~$X$ along~$A$ which resolves the indeterminacies of the rational map~$\frac{\lambda}{\mu} : X \dashrightarrow \mathbb{P}^1$ into a map~$f : Y \to \mathbb{P}^1$.
The fibre~$f^{-1}(t)$ is isomorphic to~$X_t \eqdef X \cap H_t$.
The defining equation for $X_t$ when $t\ne \infty$ is
\begin{equation}
P_t = t^4(2x+3y+z)^4 + x^4+y^4+z^4\,.
\end{equation}
The map $f$ has 36 critical values $t_1, \dotsc, t_{36}$. We chose a basepoint $b$ and a value which will serve as $\infty$, both regular.

\subsection{Computing cohomology}
The primitive cohomology $PH^2(X)$ is computed thanks to the Griffiths--Dwork reduction (see Section \ref{sec:de-rham-middle}). $PH^2(X)$ has rank $21$ and a basis is given by the residues of rational forms 
\begin{equation}
\frac{1}{P}\Omega_3, \frac{A_1}{P^2}\Omega_3, \dots, \frac{A_{19}}{P^2}\Omega_3, \frac{w^2 x^2 y^2 z^2}{P^3}\Omega_3 \in H^{3}(\Proj^2\setminus X)\,,
\end{equation}
where~$A_1,\dotsc,A_{19}$ are all the monomials of degree~4 in~$w, x, y, z$ with exponents at most 2,
and~$\Omega_3 = w\ud x\ud y\ud z - x\ud w\ud y \ud z + y\ud w\ud x\ud z - z\ud w\ud x\ud y$ is the volume form of $\Proj^3$.
The 21 monomials~1, $w^2x^2y^2z^2$ and~$A_1,\dotsc,A_{19}$
are all the monomials whose degree is a multiple of~4 and that are not divisible by the leading term of any element of the Jacobian ideal of~$P$, which is in this case, the monomial ideal~$\langle w^3, x^3, y^3, z^3\rangle$.

Similarly, a basis of~$\mathcal{H}$, defined as the space of sections of~$\PH_1(X_t)$ (which, since~1 is odd, is just~$H_1(X_t)$) is given by the residues
of the forms
\begin{equation}\label{eq:12}
  \frac{x}{P_t}\Omega_2, \frac{y}{P_t} \Omega_2, \frac{z}{P_t}\Omega_2, \frac{z^5}{P_t^2}\Omega_2,
  \frac{yz^4}{P_t^2}\Omega_2, \text{ and } \frac{xz^4}{P_t^2}\Omega_2.
\end{equation}

\subsection{The action of monodromy on $H_1(X_b)$, thimbles, and recovering $H_2(Y)$}\label{monodromy_thimbles_H2Y}
As $X_b$ is a smooth quartic curve, it has genus~3 and the homology group~$H_1(X_b)$ is free of rank~6.
We assume we have a (primitive) period matrix of~$H_1(X_b)$
given in the basis~\eqref{eq:12} for~$\Hdr^1(X_b)$
and some basis~$\eta_1, \dots, \eta_6$ of $H_1(X_b)$ which needs not be specified.
We first aim at computing the action of $\pi_1(\Proj^1\setminus\{t_1\dots, t_{36}\}, b)$ on $H_1(X_b)$.

First we compute the simple direct loops $\ell_1, \dots, \ell_{36}$ around the critical values $t_1, \dots, t_{36}$, such that the composition $\ell_{36}\dots\ell_1$ is the indirect loop around $\infty$. 
Then for each $i$ we may compute the monodromy matrix $M_i\in GL_6(\Z)$ of the action of monodromy along $\ell_i$ on $H_1(Y_b)$ in the basis~$\eta_1, \dots, \eta_6$ (see Section~\ref{sec:comp-monodr-matr}).
For instance, we find

\begin{equation}M_1 =
  \begin{bmatrix}
1 & 0 & 1 & 1 & 2 & 2 \\
0 & 1 & -1 & -1 & -2 & -2 \\
0 & 0 & 2 & 1 & 2 & 2 \\
0 & 0 & -1 & 0 & -2 & -2 \\
0 & 0 & 1 & 1 & 3 & 2 \\
0 & 0 & -1 & -1 & -2 & -1
  \end{bmatrix} = I_6+\begin{bmatrix} 1\\-1\\1\\-1\\1\\-1 \end{bmatrix}
  \cdot \begin{bmatrix}0&0&1&1&2&2 \end{bmatrix}\,.
\end{equation}
This decomposition is the one of Equation~\eqref{eq:decomp_mono_mat}.
We choose a generator $d_i\in H_1(Y_b)$ of the image of $M_i-I$ (the choice is up to a sign). 
This is the vector of the coordinates of the vanishing cycle $\delta_i$ at $t_i$ in the basis of $H_1(X_b)$. 
We have for example $d_1 = (1,-1,1,-1,1,-1)$. 
We also pick a permuting cycle, i.e.\ a preimage $p_i$ of $d_i$ through $M_i-I$, so that $d_i = M_ip_i-p_i$. For instance $p_1 = (0,0,1,0,0,0)$.
We then have an explicit understanding of the thimble $\Delta_i \in H_2(Y_+, X_b)$ as the extension $\tau_{\ell_i}(p_i)$ of $p_i$ along $\ell_i$. 
These thimbles freely generate $H_2(Y_+, X_b)$, and we have the $36\times 6$ integer matrix $B$ of the border map
\begin{equation}
  \tilde\delta: \begin{cases} H_2(Y_+, Y_b)\to PH_1(Y_b)\\ \Delta_i\mapsto \delta_i \end{cases} \,,
\end{equation}
as per \eqref{eq:9}. This matrix has full column rank, and its kernel gives us a basis for $H_2(Y_+)/H_2(X_b)$, which has rank~30.

\begin{figure}
\hspace{-15em}
    \resizebox{0.75\textwidth}{!}{
      \[\left(\begin{array}{rrrrrrrrrrrrrrrrrrrrrrrrrrrrrrrrrrrr}
1 & 1 & 1 & 1 & 0 & \llap{$-$}2 & 1 & 1 & \llap{$-$}2 & 1 & 1 & 3 & 0 & \llap{$-$}2 & 2 & \llap{$-$}3 & 1 & 1 & 1 & 1 & 1 & \llap{$-$}3 & 0 & 1 & 0 & 1 & 0 & 0 & 1 & \llap{$-$}2 & 1 & 1 & 0 & 1 & 0 & 0 \\
\llap{$-$}1 & \llap{$-$}1 & \llap{$-$}2 & \llap{$-$}1 & 0 & 2 & \llap{$-$}1 & 0 & 1 & \llap{$-$}1 & 0 & \llap{$-$}2 & 1 & 1 & 0 & 1 & 0 & 0 & 0 & 0 & \llap{$-$}1 & 1 & 2 & 1 & 1 & 0 & 0 & 0 & \llap{$-$}1 & 1 & \llap{$-$}2 & \llap{$-$}1 & 0 & 0 & 1 & 0 \\
1 & \llap{$-$}1 & 1 & 0 & 2 & 0 & 0 & 0 & 0 & \llap{$-$}1 & \llap{$-$}1 & 1 & 1 & \llap{$-$}1 & 1 & \llap{$-$}1 & 1 & 1 & 1 & 1 & 1 & \llap{$-$}2 & 2 & 0 & 2 & 2 & 2 & 1 & 1 & \llap{$-$}1 & \llap{$-$}1 & \llap{$-$}1 & 1 & \llap{$-$}1 & \llap{$-$}1 & 1 \\
\llap{$-$}1 & \llap{$-$}1 & \llap{$-$}2 & 0 & 1 & 2 & \llap{$-$}1 & 1 & 0 & \llap{$-$}1 & 0 & \llap{$-$}1 & 2 & \llap{$-$}1 & 2 & \llap{$-$}1 & 1 & 1 & 1 & 1 & \llap{$-$}1 & \llap{$-$}1 & 4 & 2 & 2 & 1 & 1 & 1 & \llap{$-$}1 & 0 & \llap{$-$}3 & \llap{$-$}1 & 0 & 0 & 1 & 1 \\
1 & 2 & 2 & 0 & \llap{$-$}1 & \llap{$-$}2 & 2 & \llap{$-$}1 & 0 & 2 & 0 & 1 & \llap{$-$}3 & 1 & \llap{$-$}2 & 1 & \llap{$-$}1 & \llap{$-$}1 & \llap{$-$}1 & \llap{$-$}1 & 1 & 0 & \llap{$-$}4 & \llap{$-$}1 & \llap{$-$}2 & \llap{$-$}1 & \llap{$-$}1 & \llap{$-$}1 & 1 & \llap{$-$}1 & 3 & 1 & 0 & 0 & 0 & \llap{$-$}1 \\
\llap{$-$}1 & \llap{$-$}1 & \llap{$-$}1 & 0 & 0 & 1 & \llap{$-$}1 & 0 & 0 & \llap{$-$}1 & 0 & \llap{$-$}1 & 1 & 0 & 0 & 0 & 0 & 0 & 0 & 0 & \llap{$-$}1 & 1 & 1 & 0 & 0 & 0 & 0 & 0 & \llap{$-$}1 & 1 & \llap{$-$}1 & 0 & 0 & 0 & 0 & 0
\end{array}\right)
      \]
  }
  \caption{The $6\times 36$ matrix $B$ of the border map $\tilde\delta: H_2(Y_+, Y_b)\to PH_1(Y_b)$. Each column corresponds to the coordinates of a vanishing cycle at a critical point in the undetermined basis of $PH_{1}(X_b)$.}
\end{figure}

In order to recover $\mathcal{T}(Y)$, we need to quotient by the extensions of cycles in $H_1(Y_b)$ along the loop around $\infty$, which we recall is simply the composition $\ell_{36}\cdot\dots\cdot\ell_1$.
The matrix $T_i$ of the extension map $\tau_{\ell_i}: H_1(X_b)\to H_2(Y, X_b)$ in the bases $\beta_1, \dots, \beta_6$ of $H_1(Y_b)$ and $\Delta_1, \dots, \Delta_{36}$ of $ H_n(Y, X_b)$ is given by equation \eqref{eq:extension_matrix}. For instance, we have
\begin{equation}
T_1 = \begin{bmatrix}
0&0&1&1&2&2\\
0&0&0&0&0&0\\
\vdots&&&&&\vdots\\
0&0&0&0&0&0
\end{bmatrix}\in \Z^{36\times6}\,.
\end{equation}
Using equation \eqref{eq:infinity_cycles}, we may then compute the matrix $T_{\infty}$ of the extension map $\tau_\infty : H_1(X_b)\to H_2(Y, X_b)$.

\begin{figure}
\hspace{-15em}
    \resizebox{0.75\textwidth}{!}{
      \[\left(\begin{array}{rrrrrrrrrrrrrrrrrrrrrrrrrrrrrrrrrrrr}
0 & 1 & 0 & 1 & 0 & 0 & 0 & 0 & 1 & 0 & \llap{$-$}1 & 0 & 1 & 0 & 0 & 0 & 0 & 0 & 0 & 0 & 0 & 1 & 0 & 0 & 0 & 1 & 0 & \llap{$-$}1 & 0 & \llap{$-$}1 & 0 & 0 & 0 & 1 & 0 & 0 \\
0 & 1 & \llap{$-$}1 & 2 & 0 & 0 & 0 & 0 & 0 & 1 & \llap{$-$}2 & 0 & 1 & \llap{$-$}1 & \llap{$-$}1 & 1 & 1 & 1 & 1 & 1 & 1 & 2 & 0 & \llap{$-$}1 & 0 & 1 & 0 & \llap{$-$}2 & 0 & \llap{$-$}1 & 0 & \llap{$-$}1 & 1 & 2 & 1 & 0 \\
1 & 0 & \llap{$-$}1 & 1 & \llap{$-$}1 & 0 & 0 & 0 & 0 & 1 & \llap{$-$}1 & \llap{$-$}1 & 0 & 0 & 0 & 0 & 1 & 1 & 1 & 1 & 0 & 1 & 0 & \llap{$-$}1 & \llap{$-$}1 & 0 & 1 & \llap{$-$}1 & 0 & \llap{$-$}1 & 0 & \llap{$-$}1 & 1 & 1 & 0 & 0 \\
1 & \llap{$-$}1 & 0 & 0 & 0 & \llap{$-$}1 & \llap{$-$}1 & 0 & 0 & 0 & 1 & \llap{$-$}1 & \llap{$-$}1 & 1 & 0 & 0 & 1 & 1 & 1 & 1 & 0 & 0 & 0 & 0 & \llap{$-$}1 & \llap{$-$}1 & 0 & 0 & \llap{$-$}1 & \llap{$-$}1 & 0 & 0 & 0 & \llap{$-$}1 & 0 & 1 \\
2 & 0 & \llap{$-$}1 & 1 & 0 & \llap{$-$}1 & \llap{$-$}1 & 1 & 0 & 1 & 0 & \llap{$-$}2 & 0 & 0 & \llap{$-$}1 & 1 & 2 & 2 & 2 & 2 & 0 & 1 & 1 & \llap{$-$}2 & \llap{$-$}2 & 0 & 0 & \llap{$-$}2 & \llap{$-$}1 & \llap{$-$}2 & 1 & \llap{$-$}1 & 1 & 0 & 1 & 1 \\
2 & 1 & 0 & 0 & 0 & \llap{$-$}1 & 0 & 2 & 1 & 1 & 0 & \llap{$-$}2 & 1 & 0 & 0 & 1 & 1 & 1 & 1 & 1 & \llap{$-$}1 & 1 & 2 & \llap{$-$}2 & \llap{$-$}2 & 1 & 0 & \llap{$-$}2 & \llap{$-$}1 & \llap{$-$}2 & 1 & 0 & 1 & 0 & 0 & 0
\end{array}\right)
      \]
  }
  \caption{The transpose of the $36\times 6$ matrix $T_\infty$ of the extension map $\tau_\infty: H_1(Y_b)\to H_2(Y_+, Y_b)$. Each line corresponds to the coordinates of an extension along the equator in the basis of thimbles $\Delta_1, \dots, \Delta_{36}$.}
\end{figure}

We may then compute a supplement (as a $\Z$-module) of the image of $T_\infty$ in the kernel of $B$, which has rank $36-6-6=24$.
This gives a description of $\mathcal{T}(Y)$ as integer linear combinations of thimbles, given as $24$ vectors of $\Z^{36}$.
We may compute a basis $e_1, \dots, e_{24}$ of this space.
For instance we compute $e_2 = \Delta_2 - \Delta_{31} - \Delta_{35} - \Delta_{36}$.

\subsection{Integrating forms}
Let $\pi:Y\to X$ denote the canonical projection.
As we know the periods of $X_b$, we may compute the integral of the pullback of a primitive cohomology form $\res\omega_j \in \PHdr^2(X)$ along the thimbles
$\int_{\Delta_i}\pi^*\res\omega_j$ using methods detailed in Section \ref{sec:thimble_integration}. 
To recover the integral along an extension, it is sufficient to take the corresponding linear combinations of integrals along thimbles. 
For instance
\begin{equation}
\begin{split}
\int_{e_2}\pi^*\res\omega_j &= \int_{ \Delta_{2}}\pi^*\res\omega_j - \int_{ \Delta_{33}}\pi^*\res\omega_j - \int_{ \Delta_{35}}\pi^*\res\omega_j - \int_{ \Delta_{36}}\pi^*\res\omega_j\\
&=i1.718796454505093\dots \hspace{1em}\text{ with 139 digits of precision.}
\end{split}
\end{equation}
This allows us to recover the full pairing $\mathcal{T}(Y)\times \PHdr^2(X) \to \C$.

\subsection{Recovering $PH_2(X)$}
To recover $PH_2(X)$, we need to remove the $3$ differences of blowup cycles in $H_2(Y)$, i.e.\ $E_i-E_1$ for $i\in \{2,3,4\}$.
To identify them in $\mathcal{T}(Y)$ we can simply take the right kernel of the $21\times24$ period matrix $\left(\int_{e_j}\pi^*\res\omega_i\right)$.


\section{The lattice structure on $H_n(X)$}
\label{sec:further-algorithms}

In Section~\ref{sec:algorithms} we showed how to compute the matrix of the integration pairing $\PH_n(X) \times \PHdr^n(X) \to \mathbb{C}$ given a hypersurface~$X\subseteq \mathbb{P}^{n+1}$,
with respect to \emph{some} basis of~$\PH_n(X)$.
In the case of surfaces in~$\mathbb{P}^3$, which we restrict the discussion to in this section, this is enough to compute some algebraic invariants (most importantly the Picard rank), but not enough to recover finer invariants (such as the Néron-Severi lattice and the transcendental lattice).
The extra information we need is the intersection form on~$H_2(X)$, the fundamental class~$h$ of the hyperplane section~$H_2(X)$,
and the matrix of the projection~$H_2(X)\to \PH_2(X)$.
This section describes the computation of this extra information.
It also provides a rigorous way to compute the kernel of the map~$\mathcal{T}(Y) \to \PH_2(X)$, as an alternative to the heuristic method described in Section~\ref{sec:kernel-TY-to-PHX}.

For more details on how to exploit period computations with the intersection form to compute algebraic invariants of surfaces, see~\textcite{LairezSertoz_2019}.

\subsection{Exceptional divisors as thimbles}
\label{sec:removing_exc_divs}
The exceptional locus~$X'$ of the map~$X\dashrightarrow \mathbb{P}^1$ is the intersection of~$X$ with a (generic) line in~$\mathbb{P}^3$. It is therefore a set of~$d$ points~$s_1,\dotsc,s_d$, where~$d = \deg X$.
The modification~$\pi : Y \to X$ is the blowup of~$X$ along~$X'$.
Let~$E_1,\dotsc,E_d \subset Y$ denote the $d$ components of the exceptional divisor, that is~$E_k = \pi^{-1}(s_k)$. They are all isomorphic to~$\mathbb{P}^1$, and linearly independent.
The~$E_k$ define classes in~$H_2(Y)$ and we have the exact sequence \parencite[(3.1.2)]{Lamotke_1981}
\begin{equation}
  0 \to \bigoplus_{k=1}^d \mathbb{Z} E_k \to H_2(Y) \to H_2(X) \to 0.
\end{equation}
For any two~$E_k$ and~$E_j$, the intersection~$(E_k - E_j)\cap X_b$ is the difference of two points, which is homologous to~0 in $H_0(X_b)$.
Therefore, by Theorem~\ref{lem:TY}, the homology class of~$E_k - E_j$ comes from a uniquely determined element in~$\mathcal{T}(Y)$.
In particular, we obtain the exact sequence
\begin{equation}
  0 \to \bigoplus_{k=2}^d \mathbb{Z} (E_k - E_1) \to \mathcal{T}(Y) \to \PH_2(X) \to 0.
\end{equation}

We can compute the image of~$E_k - E_1$ in~$\mathcal{T}(Y)$ in terms of the Lefschetz thimbles as follows.
Let $p_1, \dots, p_r$ be non-intersecting paths in $\Proj^1$ connecting $b$ to the critical points $t_1, \dots, t_r$ respectively.
Consider $T= \cup_{i=1}^r p_i \subset \Proj^1$.
It defines an oriented tree covering the critical points of $f$.
Finally let $U = \Proj^1\setminus T$. It is a simply connected subset of $\Proj^1\setminus\Sigma$, above which~$f$ has no critical point.
From Thom's isotopy lemma \parencite{Mather_2012} applied to the pair $(Y, \pi^{-1}(X'))$
we obtain a trivialisation of the fibration $f^{-1}(U) \to U$ where the points $s_k$ are fixed in the fibres. In other words,
there is a homeomorphism~$\phi : f^{-1}(U) \to X_{b'} \times U$, where~$b'\in U$, such that the following diagram is commutative
\begin{equation}
  \label{eq:thom_iso}
  \begin{tikzcd}
    X' \times U \arrow[d, "\iota"]\arrow[dr, "\iota_2"]& \\
    f^{-1}(U)\arrow[r, "\phi"]\arrow[d, "f"]& X_b\times U\arrow[dl, "p_2"]\,,\\
    U
  \end{tikzcd}
\end{equation}
where $\iota$ and $\iota_2$ are inclusions.

Take a $1$-chain $\alpha_k$ in $X_{b'}$ such that $\partial \alpha_k = s_k- s_1$ (that is a path connecting $s_k$ to $s_1$ in~$X_{b'}$).
For $u\in U$, let~$\alpha_k(u) = \phi^{-1}( \left\{ u \right\}\times \alpha_k)$.
Note that
\begin{equation}\label{eq:22}
  \partial \alpha_k(u) =\iota_*(s_k,u) - \iota_*(s_1, u) = s_k - s_1.
\end{equation}

For $t\in T$ not a vertex, that is~$t \neq b$ and~$t\not\in \Sigma$, define $\alpha_k(t^+)$ and $\alpha_k(t^-)$ as the left and right limit of $\alpha_k(u)$ as $u\to t$ for $t\in T$ (where the direction is given by the orientation of $T$).
Since $s_k$ and $s_1$ are fixed,  $\partial(\alpha_k(t^+) - \alpha_k(t^-)) = 0$, and this chain defines a cycle $v_k(t)\in H_1(X_t)$.
Let~$v_{kj} \in H_{1}(X_b)$ be the limit of~$v_k(t)$ as $t \to b$ along the~$j$-th branch of~$T$.

\begin{lemma}\label{lem:path-monodromy}
  With the above notations, the cycle~$v_{kj} \in H_{1}(X_b)$ is a multiple of~$\delta_j$, the vanishing cycle associated to the critical value~$t_j$.
\end{lemma}

\begin{proof}
  Consider a small enough ball~$B$ in~$Y$ around the critical point associated to the critical value~$t_j$.
  Let~$t \in T$ be very close to~$t_j$.
  The paths~$\alpha_k(t^+)$ and $\alpha_k(t^-)$ defines an element of~$H_1(X_t, X_t \setminus B)$.
  Following a loop from~$t$ around~$t_i$ transforms~$\alpha_k(t^-)$ into~$\alpha(t^+)$.
  By \parencite[(6.5.1)]{Lamotke_1981},
  the extension of~$\alpha_k(t^+)$ along a loop around~$t_j$ gives a multiple of the $j$-th thimble.
 In particular the boundary of this extension, which is just~$\alpha_k(t^+) - \alpha_k(t^-)$, is a multiple of the vanishing cycle~$\delta_j$.
  Since~$v_{kj}$ is the deformation of $\alpha_k(t^+) - \alpha_k(t^-)$ along the $j$-th branch, we obtain the claim.
\end{proof}

Let~$m_{kj} \in \mathbb{Z}$ such that~$v_{kj} = m_{kj} \delta_j$.

\begin{lemma}\label{lem:diff-exc-div}
  With the above notations, $E_k - E_1 = \sum_{j=1}^r m_{kj} \Delta_j$ in~$\mathcal{T}(Y)$.
\end{lemma}

\begin{proof}
  Consider in~$Y$ the 3-chain
  \begin{equation}
    B = \overline{\phi^{-1}(\alpha_k\times U)} = -\overline{\cup_{u\in U} \alpha_k(u)}.
  \end{equation}
  The border~$\partial B$ defines an element of~$H_2(Y)$.
  It decomposes into a part coming from the border of~$\alpha_k$ and another coming from the border of~$U$.
  The part coming from~$\partial \alpha_k$ is the closure of all~$\partial \alpha_k(u)$, that is~$E_k - E_1$, by~\eqref{eq:22}.
  The part coming from~$\partial U$ is $\cup_{t \in T} (\alpha_k(t^+) - \alpha_k(t^-))$.
  For a given edge~$p_j$ of~$T$, the union~$\cup_{t \in p_j} (\alpha_k(t^+) - \alpha_k(t^-))$
  is the extension along~$p_j$ of the cycle~$v_{kj}$.
  Since~$v_{kj}$ is a multiple of the vanishing cycle $m_{kj}\delta_j$, the boundary of~$\cup_{t \in p_j} (\alpha_k(t^+) - \alpha_k(t^-))$ is $m_{kj}\delta_j$.
  In particular this extension is a multiple of the Lefschetz' thimble~$\Delta_j$ with the same factor.
  
\end{proof}

An illustration of this construction is given in \figref{exceptional_divisors}. 
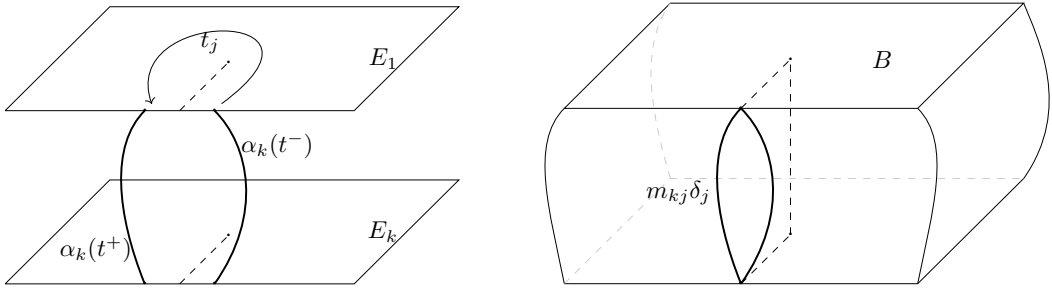
\begin{figure}[tp]
  \centering
  \begin{subfigure}[t]{0.42\linewidth}
  \resizebox{\columnwidth}{!}{\begin{tikzpicture}
	
	\draw[very thin] (-2.5, 2.5) -- (2.5,2.5) -- (4,4) node[midway, left] {$E_1$} -- (-1, 4) -- (-2.5,2.5);
	\draw[very thin] (-2.5, 0) -- (2.5,0) -- (4,1.5) node[midway, left] {$E_k$} -- (-1, 1.5) -- (-2.5,0);
	
	\draw[dashed] (0,0) -- (0.7, 0.7) node {$\cdot$};
	\draw[dashed] (0,2.5) -- (0.7, 3.2) node {$\cdot$} node[above left] {$t_j$};
	
	\draw [very thin, ->] (0.6,2.6)  ..controls (2.5,4) and (-0.9,4).. (-0.4, 2.6);

	\draw [thick] (-0.5,0) node {$\large\cdot$} ..controls (-0.9, 1) and (-1, 2) .. (-0.5,2.5) node {$\large\cdot$};
	\draw (-1.2, 0.5) node {$\alpha_k(t^+)$};
	\draw [thick] (0.5,0) node {$\large\cdot$} ..controls (1.2, 1) and (1, 2) .. (0.5,2.5) node {$\large\cdot$};
	\draw (1.4,2) node {$\alpha_k(t^-)$};
	
\end{tikzpicture}}

    \caption{A 1-chain $\alpha_k$ connecting $s_0$ to $s_1$ in one of the fibres may have monodromy around the critical values $t_j$. Above $T$ the chain $\alpha_k(t^+)-\alpha_k(t^-)$ has no boundary, and thus represents a $1$-cycle.}
  \end{subfigure}
  \hspace{1em}
  \begin{subfigure}[t]{0.50\linewidth}
  \resizebox{\columnwidth}{!}{\begin{tikzpicture}
	
	\draw[very thin] (-2.5, 2.5) -- (2.5,2.5) -- (4,4) -- (-1, 4) -- (-2.5,2.5);
	\draw[very thin] (-2.5, 0) -- (2.5,0) -- (4,1.5) ;
	
	\draw[very thin, gray!40, dashed] (-2.5, 0) -- (-1,1.5) -- (4,1.5) ;
	
	\draw[dashed] (0,0) -- (0.7, 0.7) node {$\cdot$} ;
	\draw[dashed] (0,2.5) -- (0.7, 3.2) node {$\cdot$};
	\draw[dashed] (0.7, 0.7) -- (0.7, 3.2);

	\draw [thick] (0,0) node {$\large\cdot$} ..controls (-0.4, 1) and (-0.5, 2) .. (0,2.5) node {$\large\cdot$};
	\draw [thick] (0,0) node {$\large\cdot$} ..controls (0.7, 1) and (0.5, 2) .. (0,2.5) node {$\large\cdot$};
	
	\draw [thin] (-2.5,0)  ..controls (-2.7, 1) and (-3, 2) .. (-2.5,2.5) ;
	\draw [thin] (2.5,0)  ..controls (2.7, 1) and (3, 2) .. (2.5,2.5) ;
	\draw [thin] (4,1.5)  ..controls (4.7, 2.5) and (4.2, 3.5) .. (4,4) ;
	
	\draw [thin, gray!40, dashed] (-1,1.5) ..controls (-1.2, 2.2) and (-1.5, 3.3) .. (-1,4) ;
	
	\draw (-0.3, 1.3) node[left] {$m_{kj}\delta_j$};
	\draw (2, 3.2) node {$B$};
	
\end{tikzpicture}}

    \caption{We obtain a $3$-chain $B$ by extending $A$ to all of $\mathbb{P}^1 \setminus T$. The boundary of $B$ is the sum of $E_k - E_1$ with the $2$-chain $\cup_{t\in T} (\alpha_k(t^+)-\alpha_k(t^-))$.
    Since the intersection of this $2$-chain with $p_j$ has boundary only in $X_b$, $\alpha_k(t^+)-\alpha_k(t^-)$ is a multiple of the vanishing cycle~$\delta_j$.}
  \end{subfigure}

  \caption{The action of monodromy on relative homology and total transforms}
  \label{exceptional_divisors}
\end{figure}

\subsubsection{Monodromy action on relative homology}

Consider the relative homology space~$H_1(X_b, X')$, where~$X'$ is the indeterminacy locus of~$f$.
As in the previous section, let~$\alpha_k$ be a path from~$s_1$ to~$s_k$ in~$X_b$.
It follows directly from the long exact sequence of relative homology of the pair $(X_b, X')$ that
\begin{equation}
\label{relative_homology_decomp}
H_1(X_b, X')\simeq H_1(X_b)\oplus \bigoplus_{i=2}^d \mathbb{Z} \alpha_k.
\end{equation}
Following the construction in the previous section, we have a monodromy action of~$\pi_1(\mathbb{P}^1\setminus \Sigma, b)$
on~$H_1(X_b, X')$ which extends the monodromy action on~$H_1(X_b)$.

In view of Lemma~\ref{lem:diff-exc-div}, the problem of computing the exceptional divisors in~$\mathcal{T}(Y)$ reduces to the computation of the coefficients~$m_{kj}$,
which are determined by the monodromy action on~$H_1(X_b, X')$.

\begin{lemma}
  The integers~$m_{kj}$ in Lemma~\ref{lem:diff-exc-div} satisfy
  \[ {\ell_j}_* \alpha_k = \alpha_k + m_{kj} \delta_j. \]
\end{lemma}

\begin{proof}
  This is simply a reformulation of the definition of~$m_{kj}$.
\end{proof}

The method described in Section~\ref{sec:comp-monodr-matr-1} to compute the matrices of the monodromy action on~$H_1(X_b)$ extends to the relative case to compute the action on~$H_1(X_b, X')$.

We choose a basis~$\omega_1,\dotsc,\omega_r$ of~$\mathcal{H}$
in the form
\begin{equation}
  \omega_i = \res \frac{B_i}{P_t^2} \Omega_2.
\end{equation}
Recall the period matrix~$\Pi(t)$ defined by the coefficients
\begin{equation}
  \Pi_{ij}(t) = \int_{\eta_j(t)} \omega_i(t) = \int_{T(\eta_j(t))} \frac{B_i}{P_t^2} \Omega_2,
\end{equation}
where~$T : H_1(X_t) \to H_2(\mathbb{P}^2\setminus X_t)$ is the Leray tube map.
Recall that~$\Pi(t)$ satisfies a differential equation~$\Pi'(t) = A(t) \Pi(t)$, see~\eqref{eq:1},
and that the monodromy action on the solution space of this differential equation
is dual to the monodromy action on~$H_1(X_b)$.

We can extend the matrix~$\Pi(t)$ with integrals related to the paths~$\alpha_k(t)$.
For each~$k$, we define a Leray tube~$T(\alpha_k(t))$ around~$\alpha_k(t)$ as follows.
For each point~$p$ of~$\alpha_k(t)$, we choose, continuously with respect to~$p$, a line~$L_p$ in~$\mathbb{P}^2$ passing through~$p$
and not tangent to~$X_t$.
Then, for $\varepsilon$ small enough, we define $T(\alpha_k(t))$ as the union over all points~$p\in \alpha_k(t)$ of the $\varepsilon$-circle in~$L_p$ with center~$p$.
Up to homotopy, $T(\alpha_k(t))$ only depends on the choice of~$L_{s_1}$ and~$L_{s_k}$, which determine the border.
Indeed, for each~$p$, the space of possible lines~$L_p$ is contractible, so the choice of~$L_p$ is irrelevant.
We assume that~$L_{s_1}$ and~$L_{s_k}$ are fixed, so that the border of~$T(\alpha_k(t))$ is constant.

Let
\begin{equation}
  \label{eq:23}
  \Theta_{ik}(t) = \int_{T(\alpha_k(t))} \frac{B_i}{P_t^2} \Omega_2.
\end{equation}
It is an analytic function of~$t$, in a neighbourhood of~$b$. Indeed, if~$t$ is close enough to~$b$, then $T(\alpha_k(t))$ deforms into~$T(\alpha_k(b))$ in~$X_t^\complement$, with fixed boundary.
So the integral~\eqref{eq:23} may be taken over the fixed domain~$T(\alpha_k(b))$.
Since the integrand depends analytically on the parameter~$t$, this shows that~$\Theta_{ik}(t)$ depends analytically on~$t$.
By following the deformation of the path~$\alpha_k(t)$, the function~$\Theta_{ik}$ extends meromorphically on any simply connected open subset of the complex plane
avoiding the singular values~$\Sigma$. (There may be poles at points~$t$ where~$L_{s_1}$ or~$L_{s_k}$ are tangent to~$X_t$.)
After extending~$\Theta_{ik}(t)$ over a loop~$\ell_j$, we obtain a new determination~${\ell_j}_* \Theta_{ik}$ which satisfies
\begin{equation}
  \label{eq:24}
  {\ell_j}_* \Theta_{ik}(t) = \int_{T(\alpha_k(t) + v_{kj}(t))} \frac{B_i}{P_t^2} \Omega_2
  = \Theta_{ik}(t) + m_{kj} \int_{\delta_j(t)} \omega_i(t).
\end{equation}
In particular, the monodromy action on the functions~$\Theta_{ik}$ determine the coefficients~$m_{kj}$, and therefore the monodromy action on~$H_1(X_b, X')$.

It remains to see that the~$\Theta_{ik}$ are solutions of a differential system, so that the monodromy action can be recovered by numerical integration.
Indeed,
by definition of~$\mathcal{H}$,
there are some~$\beta_i$ 1-forms on~$X_t^\complement$ such that
\begin{equation}
  \label{eq:25}
  \frac{\partial}{\partial t} \frac{B_i}{P_t^2} \Omega_2 = \sum_{j} a_{ij}(t) \frac{B_j}{P_t^2} \Omega_2 + \ud \beta_i.
\end{equation}
After integrating over~$T(\alpha_k(t))$,
we obtain
\begin{equation}
  \Theta_{ik}'(i) = \sum_j a_{ij} (t) \Theta_{ik}(t) + \int_{T(\alpha_k(t))} \ud \beta_i.
\end{equation}
The boundary of~$T(\alpha_k(t))$ is two circles, one in~$L_{s_k}$ and one in~$L_{s_1}$ (with a minus sign), so by Stokes' and Cauchy's formulae,
\begin{equation}
  \int_{T(\alpha_k(t))} \ud \beta_i = \int_{\partial T(\alpha_k(t))} \beta_i =  \res_{s_k}( \beta_i|_{L_{s_k}} ) - \res_{s_1}( \beta_i|_{L_{s_1}} ).
\end{equation}
Since~$s_1$ and~$s_k$ are fixed, this is simply a rational function in~$t$, which we denote~$R_{ik}(t)$, defining a matrix~$R \in \mathbb{C}(t)^{r \times (d-1)}$.

So we can consider the following differential system, of dimension~$r+d-1$,
\begin{equation}
  Z' = \left(\begin{array}{c|c} A& R \\\hline 0 & 0 \end{array}\right) Z,
\end{equation}
which admits the fundamental solution
\begin{equation}\label{eq:28}
 \tilde \Pi = \left(
   \begin{array}{c|c}
     \Pi & \Theta \\\hline
     0 & \mathbf{1}
   \end{array}
 \right)\,.
\end{equation}
The monodromy of this differential system is conjugate to the monodromy action on~$H_1(X_b, X')$.
Namely, considering the matrix~$\operatorname{Mat}(\ell_{*})$ of the action of a path~$\ell$ in the basis $\eta_1,\dotsc,\eta_r, \alpha_2,\dotsc,\alpha_d$ of~$H_1(X_b, X')$,
we have similarly to~\eqref{eq:27}
\begin{equation}
  \operatorname{Mat}(\ell_{*}) = \tilde \Pi(b)^{-1} \cdot \ell_{*} \tilde \Pi(b) = \tilde \Pi(b)^{-1} \tilde\Lambda_{\ell} \tilde \Pi(b),
\end{equation}
where~$\tilde \Lambda_\ell$ is the transition matrix associated to the system~\eqref{eq:28}.
This gives the desired algorithm for computing the monodromy action on relative homology.

\begin{remark}
  The order of the differential operators that need to be integrated is larger than those of Section~\ref{sec:comp-monodr-matr}, so the computation is expensive in practice, even for quartic surfaces.
  An alternative approach is to track explicitly the deformation of the paths~$\alpha_k(t)$  by following the movement of the critical values of the fibration $X_t\dashrightarrow \Proj^1$ induced by a fixed projection $\Proj^2\dashrightarrow\Proj^1$ when $t$ changes.
  Indeed, we may choose the fibration to be such that $X' = X_b \cap H'_a$ where $\{H'_t\}$ is the hyperplane pencil of the fibration of $X_b$, and thus $H_{1}(X_b, X') = H_{1}(X_b, X_b\cap H'_a)$.
  It then becomes apparent that the $\tilde M_i$ represents the action of monodromy on the thimbles, which can be obtained from the braid induced by the  movement of the aforementioned critical values.
  We will not expand upon this method, which goes beyond the scope of this paper.
\end{remark}

\subsection{Computing the intersection product}
\subsubsection{The intersection product of $H_n(Y)$}
In this section we explain how to recover the intersection matrix of $H_n(Y)$ using methods similar to those of~\textcite[\S2]{shiga79} or~\textcite[\S3]{NorihikoShiga2001}. 
Let $n\le 2$ and $d$ be respectively the dimension and degree of $X$.

We begin with a lemma characterising $H_n(X)$ as a subspace of $H_n(Y)$.
\begin{lemma}
\label{lem:orthogonal_complement}
The Poincaré dual of the pullback $\pi^*$ embeds $H_n(X)$ to the orthogonal complement of $H_{n-2}(X')$ in $H_n(Y)$ isometrically (i.e.\ the intersection product is preserved).
\end{lemma}
\begin{proof}
Let $\alpha, \beta\in H^n(X)$. As $\pi^*$ preserves cup products, we have that
\begin{equation}
\langle\pi^*\alpha, \pi^*\beta\rangle = \langle\alpha, \beta\rangle\,.
\end{equation}
For $1\le i\le d$, let $\tilde E_i\in H^n(Y)$ be the Poincaré dual of the exceptional divisor $E_i$. The projection formula \parencite[p.~256, A4]{Hartshorne_1977} yields that
\begin{equation}
\pi_*(\tilde E_i\cdot \pi^*\alpha) = \pi_*\tilde E_i\cdot\alpha = 0\,,
\end{equation}
as $\pi_*\tilde E_i =0$.
Poincaré duality yields the desired result.
\end{proof}

Therefore it is relevant to compute the intersection product of $H_n(Y)$. 
When $n=2$, consider the maps
\begin{equation}
  \label{eq:inclusions}
  \begin{tikzcd}
    H_2(X_b) \arrow[r]\arrow[dr]& H_2(Y_+)\arrow[r, "\phi"]\arrow[d, "\iota_*"]& \mathcal{T}(Y) \,. \\
    &H_2(Y)&
  \end{tikzcd}
\end{equation}
Recall from Lemma \ref{torsion_free} that $H_2(X_b)$ is generated by the inclusion of a linear section $h = [X_b\cap L]$ in $X_b\subset Y$.
Furthermore, from the long exact sequence of the pair $(Y_+, X_b)$, $\phi$ is surjective.

\begin{lemma}\label{lem:induced_IP}
The intersection product on $H_2(Y_+)$ induced by $\iota_*$ induces an intersection product on $\mathcal{T}(Y)$ through $\phi$.
\end{lemma}
\begin{proof}
The inclusion $\iota:Y_+\to Y$ induces an intersection product on $H_2(Y_+)$ from that on $H_2(Y)$.
We aim to prove that this product induces a uniquely defined product on $\mathcal{T}(Y)$.
From the long exact sequence of the pair $(Y_+, X_b)$, $H_2(Y_+)$ can be identified with $\ker\delta\oplus H_2(X_b)$.

Let $\Gamma_1, \Gamma_2 \in H_2(Y_+)$ and assume that the class of $\phi(\Gamma_2)=0$.
Then $\Gamma_2 \in \im\tau_\infty\oplus h\Z$.
In particular, as $\im\tau_\infty\subset \ker\iota_*$ from the second line of \eqref{diag:bigchase}, $\iota_*\Gamma_2 = k h$ for some integer $k\in\Z$.
We thus have
\begin{equation}
	\langle\Gamma_1, \Gamma_2\rangle = k\langle\iota_*\Gamma_1, h\rangle = 0\,,
\end{equation}
as $h$ can be deformed to a fibre in the interior of the lower hemisphere $Y_-$, and thus not intersect $\iota_*\Gamma_1$.
\end{proof}

Let $\Gamma_i = \sum_j a_{ij}\Delta_{j}\in \mathcal{T}(Y)$ be two extensions, described as a linear combination of thimbles for $i=1,2$.
The previous lemma implies that the intersection product $\langle \Gamma_1, \Gamma_2\rangle$ is well defined.

In order to compute this intersection product, we may deform the geometric representatives of the thimbles for $\Gamma_2$ slightly, 
in a way that the basepoint is no longer the same for $\Gamma_1$ and $\Gamma_2$, 
as is represented in \figref{fig:intersection_product}. 
We then notice that the intersection between $\Delta_i$ and $\Delta_j'$ 
(where the latter is the aforementioned deformation of $\Delta_j$) 
is contained in at most $4$ fibres.
This means that in order to compute the intersection product $\langle \Gamma_1, \Gamma_2\rangle$, 
we may simply consider the intersection of pairs of thimbles 
(which we will also denote $\langle \Delta_i, \Delta_j\rangle$ by abuse of notation) 
and use bilinearity.  
More precisely, we see that
\begin{itemize}
\item if $i>j$, then $\langle\Delta_i, \Delta_j\rangle = 0$,
\item if $i<j$, then $\langle\Delta_i, \Delta_j\rangle = \langle \delta_i, \delta_j\rangle$,
\item if $i=j$, then $\langle\Delta_i, \Delta_j\rangle = -\langle p_i, \delta_i\rangle$,
\end{itemize}
where $\delta_i$ and $p_i$ are respectively the vanishing cycle and a permuting cycle of $\Delta_i$, i.e.\ $\delta_i = {\ell_i}_*p_i-p_i$.
We can then recover the intersection product with the formula $\langle \Gamma_1, \Gamma_2\rangle = \sum_{i,j} a_{1i} a_{2j} \langle\Delta_i, \Delta_j\rangle$.

\begin{figure}[ht]
    \centering
    \resizebox{0.5\linewidth}{!}{
    \scalefont{2}
  \resizebox{\columnwidth}{!}{\begin{tikzpicture}
	\coordinate (b)  at (0,0.5);
        \coordinate (b2) at (0,-0.5);
        
        \coordinate (t1) at (4,-5);
        \coordinate (t2) at (6.5,-4);
        \coordinate (ti) at (8.5,-1.5);
        \coordinate (tj) at (8.5,2);
        \coordinate (tr) at (4,5);
        
        \draw[fill=gray!10, draw=none] (4,0) circle (7);
        \draw[fill=gray!30, draw=none] (-0.1,0) ellipse (1.5 and 1.5);
        
        \draw[fill] (t1) node[left] {$t_1$}  circle (0.07);
        \draw[fill] (t2) node[left] {$t_2$} circle (0.07);
        \draw[fill] (ti) node[left] {$t_i$} circle (0.07);
        \draw[fill] (tj) node[left] {$t_j$} circle (0.07);
        \draw[fill] (tr) node[left] {$t_r$} circle (0.07);
        
        \draw[fill] (b) node[left] {$b$} circle (0.07);
        \draw[fill] (b2) node[left] {$b'$} circle (0.07);

        \draw[draw=none] (t2) -- (ti) node[midway]{$\dots$};
        \draw[draw=none] (ti) -- (tj) node[midway]{$\dots$};
        \draw[draw=none] (tj) -- (tr) node[midway]{$\dots$};
        
        \draw (b) .. controls ++(2,-2) and ++(-5,0) .. ($(ti) - (0,1)$) arc(-90: 90: 1) node[midway, right]{$\scriptsize\ell_i$} .. controls ++(-5,0) and ++(2,-1) ..  (b);
        \draw (b2) .. controls ++(2,1) and ++(-5,0) ..  ($(tj) - (0,1)$) arc(-90: 90: 1) node[midway, right]{$\scriptsize\ell_j$} .. controls ++(-5,0) and ++(2,2) .. (b2);
	
\end{tikzpicture}}

        }
        \caption{The intersection between two thimbles can be reduced to intersection products of cycles of the fibre at the four fibres above the intersections of $\ell_i$ and $\ell_j$.}
        	\label{fig:intersection_product}
\end{figure}

For $n=1$, this is sufficient to recover the intersection product on $H_1(Y)$.

Assume $n=2$. 
By the same argument as that of Lemma~\ref{lem:induced_IP}, $h$ is orthogonal to the image of $H_2(Y_+)$ in $H_2(Y)$.
All that remains it to compute the intersection products of an exceptional divisor, say $E_1$.
The intersection of $E_1$ with $h$ is precisely one point,  $\{s_1\}$, and thus
\begin{equation}
\langle h, E_1\rangle  = 1\,.
\end{equation}

Let $\Gamma_1, \dots, \Gamma_{s}\in H_2(Y_+)$ induce a basis of $\mathcal{T}(Y)$ through $\phi$.
Up to adding multiples of $h$, we may assume the $\Gamma_i$'s to be orthogonal to $E_1$. 
These observations allow us to recover the full intersection matrix of $H_2(Y)$. 
We synthesise these results with the following lemma.

\begin{lemma}\label{lemma:intersection_product}
There exist $\Gamma_1, \dots, \Gamma_{s}\in H_2(Y_+)$, $h\in H_2(X_b)$ and $E_1\in H_2(Y)$ such that
\begin{itemize}
\item $\phi(\Gamma_1), \dots, \phi(\Gamma_s)$ is a basis of $\mathcal{T}(Y)$,
\item $E_1$ is the exceptional divisor of the blow-up at $s_1$,
\item $h, \iota_*\Gamma_1, \dots, \iota_*\Gamma_s, E_1$ is a basis of $H_2(Y)$.
\end{itemize}
Define the coefficients $a_{ij}$ to be so that $\phi(\Gamma_i) = \sum_{j}a_{ij}\Delta_j$.
The intersection products are given by
\begin{itemize}
\item $\langle\iota_*\Gamma_i, \iota_*\Gamma_j\rangle = \sum_{k,l} a_{ik} a_{jl} \langle\Delta_k, \Delta_l\rangle$ for all $i,j$,
\item $\langle\iota_*\Gamma_i, h\rangle = 0$ for all $i$,
\item $\langle\iota_*\Gamma_i, E_1\rangle = 0$ for all $i$,
\item $\langle h, h\rangle = 0$,
\item $\langle h, E_1\rangle = 1$,
\item $\langle E_1, E_1\rangle = -1$.
\end{itemize}
The intersection matrix in the aforementioned basis of $H_2(Y)$ is therefore given by
\begin{equation}
\begin{bmatrix}
0& \textbf{0} & 1\\
\textbf{0} & M & \textbf{0} \\
1&\textbf{0} &-1
\end{bmatrix}\,,
\end{equation}
where $M$ is the matrix given coefficient-wise by $M_{ij} = \sum_{k,l} a_{ik} a_{jl} \langle\Delta_k, \Delta_l\rangle$.

\end{lemma}

\subsubsection{The intersection product of $H_2(X)$}

For surfaces, in order to recover the intersection product on $H_2(X)$, it only remains to remove the exceptional divisors $E_1, \dots, E_{d-1}$.
We begin with the following lemma which allows us to compute the intersection product of the $E_i$.
\begin{lemma}
$\langle E_i, E_j\rangle = -\delta_{ij}$ \text{(Kronecker delta)}
\end{lemma}
\begin{proof}
If $i\ne j$, $E_i$ and $E_j$ are disjoint and thus their intersection product is empty.
As $E_i$ is the exceptional divisor of a blow up at a point, $E_i^2=-1$. 
\end{proof}

\begin{lemma}
For $1\le i \le d-1$, $E_i-E_1 = \sum_{j=1}^{s} m_{ij}\Gamma_j+h$, where the $m_{ij}$ are the integers computed in section \ref{sec:removing_exc_divs}.
\end{lemma}
\begin{proof}
Per Section~\ref{sec:removing_exc_divs}, we have for every $i$ the equality
\begin{equation}
	E_i-E_1 = \sum_{j=1}^{s} m_{ij}\Gamma_j+ k_i h\,,
\end{equation}
where the $m_{ij}$'s are known integers, and $k_i\in\Z$.
Taking the intersection product with $E_1$ makes it clear that $k_i=1$.
\end{proof}
We thus have the image of $H_{n-2}(X')$ in $H_{n}(Y)$, and Lemma~\ref{lem:orthogonal_complement} allows us to conclude.

\section{Experiments}
\label{sec:experiments}

\subsection{Benchmarking}

\subsubsection{Expected timings}\label{sec:expected_timings}
Table~\ref{tab:benchmark} shows the time taken on a laptop to compute the periods of $X$ for different dimensions and degrees, as well as the obtained precision.
Each line of Table~\ref{tab:benchmark} corresponds to the computation of a single hypersurface $X = V(P)$, with $P = \sum_{m\in \mathcal M_{n,d}} a_m m$ where $\mathcal M_{n,d}$ is the set of monomial of degree $d$ of $\Q(X_0,\dots, X_n)$, and $a_m$ are random integer coefficients chosen uniformly between $-20$ and $20$.
These timings were obtained on an Apple Macbook Pro M1, using all 10 cores.\\

The column $|\Sigma|$ corresponds to the number of critical values of the fibration.
This has an impact on the number of edges on which the Picard-Fuchs operators need to be integrated, and thus on the computational cost.
More precisely, as the integration edges follow the Voronoi graph of the critical points, the number of edges is at most $3|\Sigma|-6$ \parencite[Cor. 5.2]{PreparataShamos1985}.

The column $\operatorname{rk}PH^{n}(X)$ corresponds to the number of Picard-Fuchs operators that need to be integrated to recover the full period matrix of $X$, see Section~\ref{sec:gauss-manin-conn}.

The columns $\deg \mathcal L$ and $\operatorname{ord} \mathcal L$ correspond to the degree and order of these operators.
More precisely let $\mathcal L$ be one of these Picard-Fuchs operators, corresponding to a form $\omega = \res\frac{A}{P^k}$, as per Section~\ref{sec:de-rham-middle}.
The order of $\mathcal L$ is $\operatorname{rk} {PH}^{n-1}(X_b) +1$, and is the same for all forms.
In contrast, the degree $\operatorname{deg} \mathcal L$ changes depending on the form -- therefore a range is given instead of a single value\footnote{It seems from experiments that the degree increases with the pole order $k$ of $\omega$.
As the filtration with respect to this pole order coincides with the Hodge filtration \parencite{Griffiths_1969a}, this implies that the degrees of the Picard-Fuchs operators of the holomorphic forms are the lowest. Thus the holomorphic periods are conveniently the less computationally expensive periods to compute.
}.
The numbers in the column $\deg \mathcal L$ are lower and higher bounds for this degree for a specific $P$.\\

In the case of dimension greater than 2, the exceptional divisors of the modification were not identified.
For quartic surfaces and cubic threefold, only the periods necessary to describe the Hodge structure were computed (this corresponds to respectively $1$ and $5$ forms).
This is typically what we are interested in when computing periods.
Note that, the recovered data is not sufficient to continue the induction and access higher dimensional varieties -- for this, the full period matrix is required.

\begin{table}[tp]
\begin{center}
\begin{tabular}{ cccccccc } 
 \toprule
$n$ & $d$ & $|\Sigma|$ & $\operatorname{rk} PH^{n}(X)$ & {$\deg \mathcal{L}$} & {$\operatorname{ord} \mathcal{L}$}& {Time}& Precision (dec. digits) \\
\midrule
1 &$3$& 6 &2& 10-30 & 3 &10 sec.&300 \\
2 &$3$& 12 &6& 60 & 3&3 min.&300 \\
3 &$3$& 24 &10& 110-320 & 7&8 hours*& 260 \\ 
1 &$4$& 12 &6& 30-80 & 4&8 min.&350\\
2 &$4$& 36 &21& 170-800 & 7&1 hour*&300\\
1 &$5$& 20 &12& 70-170 & 5&7 hours&270\\ 
\bottomrule
\end{tabular}\\
\vspace{0.2em}
\small{*only the periods necessary for describing the Hodge structure \\were computed (see Section~\ref{sec:expected_timings})}
\caption{Data on the computation of hypersurfaces of degree~$d$ in~$\mathbb{P}^{n+1}(\mathbb{C})$.
$|\Sigma|$ is to the number of critical values of the fibration.
$\operatorname{rk} H^{n}(X)$ is the size of the period matrix.
$\operatorname{ord} \mathcal{L}$ and~$\deg \mathcal{L}$ are the order and degree of the differential operators arising in the computation.
Time is the \emph{real} time, running on~10 cores.
Precision is the number of correct decimal digits obtained.
}
\label{tab:benchmark}
\end{center}
\end{table}

\subsubsection{Comparison with \emph{numperiods} on quartic surfaces}\label{sec:comparison}
The algorithm consists of 4 main steps:
\begin{itemize}
\item Computing the fundamental group of $\Proj^1\setminus\Sigma$;
\item Computing the Picard-Fuchs operators;
\item Computing the monodormy matrices of the Picard-Fuchs operator;
\item Recovering the action of monodormy and reconstructing the homology.
\end{itemize}
In practice, most of the time is spent on the third step. 
In the case of the computation of the periods of a quartic surface, the algorithm spends less than 1 second on step 4, around 20 seconds on step 1 and 2, and around 10 hours on step 3.

The algorithm of \textcite{Sertoz_2019} is implemented in the package \emph{numperiods}\footnote{\url{https://gitlab.inria.fr/lairez/numperiods}} \parencite{LairezSertoz_2019}.
In order to compare the efficiency of both methods, we show the time taken by \emph{numperiods} to compute the periods of K3 surfaces defined by quartic polynomials of the form $F + P$, 
where $F = x^4+y^4+z^4+w^4$ defines the Fermat quartic surface, 
and $P$ is a polynomial with $n$ monomials, with $n$ ranging from $0$ to $5$.
These examples were run on a single core, on a cluster with 500 Gb of memory, and for at most 48 hours each.
Timings for specific cases are given in Table~\ref{tab:comparison_numperiods}.
In all but five of the $100$ examples for $n\in \{4, 5\}$, the computation with \emph{numperiods} could not be carried out, either because the memory usage exceeded the allocated 500 Gb, or because the computation lasted longer than the allocated 48 hours. 

In total the CPU time for the computation of the periods of 204 quartic surfaces was measured using \emph{lefschetz-family} with an input precision of $1500$ bits.
The average time taken was 9 hours, 56 minutes and 46 seconds.
In practice, the package \emph{lefschetz-family} makes use of several cores for the computation of the monodromy matrices of the Picard-Fuchs operators, which can greatly speed up the computation.

\begin{table}[tp]
\begin{center}
\begin{tabular}{ lllll }
 \toprule
$P$& \parbox[t]{7em}{\emph{numperiods}} & \parbox[t]{7em}{\emph{lefschetz-family}} &$\operatorname{ord} \mathcal L$&$\operatorname{deg} \mathcal L$ \\\midrule
$ 0$ & < 1 s& 384 min. &--& -- \\
$ 2x^2zw$ & 4 s & 574 min. &3&4 \\
$ - 2y^3z - 4z^2w^2$ & 2 min. & 510 min. &5&38 \\
$-xyzw + 4xzw^2-2y^4$ & 25 min. & 607 min. &7&110 \\
$y^3z + z^4 + y^3w + x^2zw$ &346 min. &635 min.&14&591\\
$4xyz^2 - 5x^2zw - 4xw^3 - 4zw^3$ & > 2880 min. & 494 min. &21&?\\ 
$ - 2x^2w^2 - 4y^2w^2 - 2yzw^2 + 2yw^3 $ & > 500 Gb & 543 min. &21&?\\
$ x^4 - 4y^2z^2 - 5xz^2w + 2yz^2w + xyw^2$ & > 500 Gb & 538 min. &14&?\\
\bottomrule
\end{tabular}\\
\vspace{0.2em}
\caption{Comparison of CPU time necessary for the computation of the periods of the quartic surface defined by $x^4+y^4+z^4+w^4+P$, using \emph{numperiods} and \emph{lefschetz-family}.
  The columns~$\operatorname{ord}\mathcal{L}$ and~$\deg \mathcal{L}$ record the order of and the degree of the coefficients of the Picard-Fuchs differential equation that \emph{numperiods} integrates.
  The periods of the Fermat hypersurface are hard-coded in \emph{numperiods}, which explains the instantaneous computation for~$P = 0$.}
\label{tab:comparison_numperiods}
\end{center}
\end{table}

\subsection{An application: Picard rank of families of quartic surfaces}
In this section we explain how to use our algorithm to obtain certain algebraic invariants of quartic surfaces. 
Notably we compute the generic Picard rank of families of quartic surfaces (\S\S\ref{Bouyer} and \ref{symmetric}), we check that two quartic surfaces are isomorphic to each other (\S\ref{sec:hodge_isometry}), and give equations for quartic surfaces for each possible Picard rank (\S\ref{sec:picard_rank_examples}).
These examples allow us to test our algorithm against known results.

Given a smooth quartic surface of $\Proj^3$, we may recover its Picard rank thanks to the numerical evaluation of some of its periods.
Such a variety is a $K3$ surface. 
Its middle cohomology group $H^2(X)$ has rank $22$, and its canonical bundle is trivial: there is a unique holomorphic form~$\omega$, up to scaling.
The kernel of the map $\gamma\mapsto \int_\gamma \omega$ is a sublattice of $H_2(X)$ called the Néron-Severi group of~$X$.
The rank of this lattice is called the {\em Picard rank} (or {\em Picard number}, or {\em Néron-Severi rank}) of~$X$.

The LLL algorithm can be used to heuristically recover this kernel from high-precision numerical approximations of the periods.
This computation is not certified and can fail in two ways, in principle.
First, the algorithm may miss integer relations if the number of digits in the coefficients is not smaller than the number of significant digits computed in the periods.
Second, it may recover fake integer relations reflecting a numerical coincidence that a higher precision computation would detect.
In practice, we never observed these phenomena.
See \parencite{LairezSertoz_2019} for a discussion of these issues.

Throughout the following examples, we used the following method. 
The holomorphic $2$-form of a given quartic $X=V(P)$ can be identified as the residue of the only irreducible rational form of $H^3(\Proj^3\setminus X)$ with pole order $1$ \parencite[Eq.~8.6]{Griffiths_1969a}.
Explicitly the periods are given by $\int_\gamma \frac{\Omega_3}{P}$, with $\Omega_3$ the volume form of $\Proj^3$ and $\gamma\in H_3(\Proj^3\setminus X)$.
We may then use the steps up to~\ref{it:tubes} of Section~\ref{sec:wrapup} to compute a basis of $\mathcal{T}(Y)$.
We then compute the periods of the holomorphic form on this basis with the methods of Section~\ref{sec:thimble_integration}.
This yields 24~numerical approximations of complex numbers~$\alpha_i \in \mathbb{C}$.
Cycles on which the periods vanish induce integer linear relations between these numbers.
Of these relations, three come from the exceptional divisors of the blowup~$\pi:Y\to X$ (Section~\ref{sec:removing_exc_divs}).
As the Néron-Severi lattice is characterised by the vanishing of the periods, we may use these relations to compute it.
All in all, the Picard rank~$\rho(X)$ of~$X$ is given by
\begin{equation}
  \label{eq:7}
  \rho(X) = 22 - \dim_{\mathbb{Q}} \left\langle \alpha_1,\dotsc,\alpha_{24} \right\rangle\,.
\end{equation}
We can apply the LLL algorithm to compute $\dim_{\mathbb{Q}} \left\langle \alpha_1,\dotsc,\alpha_{24} \right\rangle$ and heuristically recover the Picard rank.

Altogether the holomorphic periods of 530 smooth quartic surfaces were computed for the results presented in this section.
The computations were run on a cluster, using 32 cores. 
The average time needed to compute the holomorphic periods of one quartic surface was around 40 minutes, ranging from 16 minutes up to 13 hours.
The median time was 28 minutes.
The computation took longer than 90 minutes for only 24 surfaces.
It should be noted that the cause of lengthy computations seems to stem from the choice of the fibration $X\dashrightarrow \Proj^1$ rather than be intrinsic to the surface itself.
Indeed, the limiting factor is the integration step, and the cases where the computation takes a lot of time seem to always be due to the integration on one single pathological edge.
In all the cases we looked at closely, picking another generic fibration reduced the computation time, to around the expected 40 minutes.
However, we have for now no way to choose a fibration that is well adapted to the computation of a given surface {\em a priori}.

\subsubsection{Families studied by Bouyer}\label{Bouyer}
In this section we numerically verify the Picard ranks of the generic elements of the families of quartic surfaces of $\mathbb P^3$ given by \textcite[Theorem~4.9]{BouyerFamilieRank}. These families are generated by polynomials of the form

\begin{equation}
\begin{split}
[A,B,C,D,E] \hspace{.5em}\eqdef \hspace{.5em}&A(x^4 + y^4 + z^4 + w^4) + Bxyzw + C(x^2y^2 + z^2w^2)\\
&\hspace{1em} + D(x^2z^2 + y^2w^2) + E(x^2w^2 + y^2z^2),
\end{split}
\end{equation}
with 5 parameters $A$, $B$, $C$, $D$ and $E$.
More precisely, the families are given respectively by polynomials $[A,B,C,D,E]$, $[A,(DE-2AC)/A,C,D,E]$, $[A,0,C,D,2AC/D]$, $[A,B(2A-B)/A,B,B,B]$ and $[A,0,C,0,0]$.
The theorem of Bouyer states that the generic Picard rank of these families are respectively $16$, $17$, $18$, $19$ and $19$.

To generate elements of these 5 families, we simply pick integers $A,B,C,D,E$ randomly in the interval $[-100,100]$
and consider the quartics defined by the above polynomials.
If these quartics are smooth we may compute the Picard rank as above. 
Otherwise we pick other values for $A,B,C,D,E$.

We checked that the values our method yields for the Picard rank coincide with the values given by the theorem for the varieties corresponding to 56 sets of values for $A,B,C,D,E$.
This gives numerical evidence of the results of~\textcite{BouyerFamilieRank}. 

We found singular examples were the Picard rank was not generic, for $[A,B,C,D,E]$ and $[A,(DE-2AC)/A,C,D,E]$ with $(A,B,C,D,E) = (49, 92, -51, 19, -51)$.
The corresponding polynomials are 
\begin{equation}
\begin{split}
[A,B,C,D,E] &= 49 x^{4} - 51 x^{2} y^{2} + 49 y^{4} + 19 x^{2} z^{2} - 51 y^{2} z^{2} + 49 z^{4} + 92 x y z w\\
&\hspace{2em} - 51 x^{2} w^{2} + 19 y^{2} w^{2} - 51 z^{2} w^{2} + 49 w^{4}\,,
\end{split}
\end{equation}
and
\begin{equation}
\begin{split}
[A,(DE-2AC)/A,C,D,E] &= 49 x^{4} - 51 x^{2} y^{2} + 49 y^{4} + 19 x^{2} z^{2} - 51 y^{2} z^{2}+ 49 z^{4} \\
&\hspace{2em} + \frac{4029}{49} x y z w - 51 x^{2} w^{2} + 19 y^{2} w^{2} - 51 z^{2} w^{2} + 49 w^{4}\,.
\end{split}
\end{equation}
The Picard ranks were 1 higher than the generic value for their respective families, i.e.\ $17$ and $18$.

\subsubsection{Picard rank of symmetric polynomials}\label{symmetric}
In this section we compute the Picard rank of families of quartic surfaces defined by a symmetric polynomial. 
The defining equation of a quartic in $\Proj^3$ is a homogeneous polynomial in $4$ variables, say $x$, $y$, $z$ and $w$. 
We consider the families of polynomials that are symmetric in some of these variables. 
Up to a permutation of the variables, there are 4 such families:
\begin{enumerate}
  \item polynomials symmetric in all the variables,
        \begin{equation*}
          \forall \sigma\in \mathfrak S_{\{x,y,z,w\}} \hspace{1em} P(x,y,z,w) = P(\sigma(x),\sigma(y),\sigma(z),\sigma(w))\,,
        \end{equation*}
  \item polynomials symmetric in three variables, say $x$, $y$ and $z$,
        \begin{equation*}
          P(x,y,z,w) = P(x,z,y,w) =P(y,x,z,w) = P(y,z,x,w) = P(z,x,y,w) = P(z,y,x,w)
        \end{equation*}
  \item polynomials where $x$ and $y$ are symmetric, as well as $z$ and $w$
        \begin{equation*}
          P(x,y,z,w) = P(y,x,z,w) = P(x,y,w,z) = P(y,x,w,z)\,,
        \end{equation*}
  \item and polynomials where $x$ and $y$ are symmetric
        \begin{equation*}
          P(x,y,z,w) = P(y,x,z,w)\,.
        \end{equation*}
\end{enumerate}
A basis of the vector space of such polynomials is given by products of elementary symmetric polynomials. 
We may thus generate {\em a priori} generic (i.e.\ with minimal Picard rank) elements of the family by picking random coefficients as in the previous section.
In practice we pick random integer coefficients in the interval $[-5, 5]$.

Doing so, we observe respectively
\begin{enumerate}
\item a Picard rank of 17 for 113 elements, 18 for one element, and 19 for one element,
\item a Picard rank of 14 for 100 elements, and 15 for one element,
\item a Picard rank of 12 for 107 elements,
\item a Picard rank of 8 for 114 elements.
\end{enumerate}
This leads us to conjecture that the generic Picard ranks of these families are respectively $17$, $14$, $12$ and $8$.

\begin{remark}
Alice Garbagnati pointed out that lower bounds on the Picard rank follow from properties of K3 surfaces with automorphisms. 
These bounds match the heuristic computation of the rank we obtained numerically.
More precisely, the automorphisms of $\Proj^3$ that permute the coordinates induce automorphisms of the K3 surfaces of the aforementioned families.
Denote $\sigma$ such an automorphism and $\sigma^*$ the induced isometry on $H_2(S)$.
Depending on the nature of the automorphism, either the transcendental lattice is included in $H_2(X)^{\sigma^*}$ (the sublattice fixed by $\sigma^*$) or $H_2(X)^{\sigma^*}$ is included in the Néron-Severi lattice.
The ranks of $H_2(X)^{G}$ for all finite group actions $G$ are known \parencite{Nikulin1979, MichelaEtal2011, Hashimoto2012, Xiao1996}, and this yields lower bounds for the Picard ranks.
However, the question of rigorously proving that the lower bounds match the Picard ranks seems to be still open.
\end{remark}

\subsubsection{Two isomorphic rank 2 smooth quartic surfaces in $\Proj^3$}
\label{sec:hodge_isometry}

\textcite[Thm. 1.4]{OguisoRank2} gives an example of two isomorphic smooth quartic $K3$ surfaces $S_1$ and $S_2$,
that are isomorphic as abstract varieties but not Cremona equivalent (i.e.\ there is no birational automorphism of $\Proj^3$ inducing an isomorphism $S_1\simeq S_2$).
Defining equations $f_1$ and $f_2$ for $S_1$ and $S_2$ are given by 
\begin{equation}
\begin{split}
f_1 &= x^3y + x^2y^2 - xy^3 + x^3z + 2x^2yz - xy^2z - y^3z + x^2z^2 - xyz^2 - 2y^2z^2 - yz^3 - z^4 \\
&\hspace{2em}+ x^3w - 2xy^2w - 2xyzw - xz^2w  + yz^2w + xyw^2 - y^2w^2 - z^2w^2 + xw^3 + yw^3
\end{split}
\end{equation}
and
\begin{equation}
\begin{split}
f_2 &= x^4 + 3x^3z - x^2yz + 3x^2z^2 - 4xyz^2 - y^2z^2 + xz^3 - 3yz^3 - x^2yw + 2xy^2w\\
&\hspace{2em}+ y^3w - 2x^2zw - 2xyzw + 3y^2zw - 3xz^2w - 3yz^2w - 2z^3w + 4xyw^2 \\
&\hspace{2em}+ 2y^2w^2+ 3yzw^2 - z^2w^2 + 2xw^3 + 2yw^3 + zw^3 + w^4
\end{split}
\end{equation}

The goal of this section is to present a method allowing one to numerically verify that these two surfaces are indeed isomorphic.\\

A {\em Hodge isometry} between two varieties $X$ and $Y$ is a morphism $X\to Y$ 
that induces an isometry of the homology groups $H_2(X)\simeq H_2(Y)$ 
 that respects the Hodge decomposition on the complexifications $H_2(X,\C)$ and $H_2(Y,\C)$.
When $X$ is a $K3$ surface, its Hodge decomposition is given by the holomorphic periods (indeed $H^{2,0}(X)$ has rank $1$ and is the complex conjugate of $H^{0,2}(X)$).

By the global Torelli theorem for $K3$ surfaces \parencite[Thm. 5.3]{Huybrechts_2016},
it is sufficient to find a Hodge isometry between the second homology groups of two $K3$ surfaces to prove that they are isomorphic.
In order to recover such an isometry, we proceed in the following way.

Using the methods presented in this paper, we compute the holomorphic period vectors of $S_i$ in some basis of homology $\gamma_1^i, \dots, \gamma_{22}^i$ of $H_2(S_i)$ for $i=1,2$.
This allows to (heuristically) recover the Néron-Severi sublattice $\NS(S_i)$, which we find has rank $2$.
The transcendental lattice $\text{Tr}(S_i)$ is then simply the orthogonal complement of $\NS(S_i)$ in $H_2(S_i)$.
The lattice $\NS(S_i)\oplus \text{Tr}(S_i)$ is a full rank sublattice of $H_2(S_i)$, which may have positive index.
In order to find a Hodge isometry, we look for an isometry between these sublattices that extends to an isometry between the full homology lattices.\\

More explicitely, let $\omega_1$ and $\omega_2$ be the holomorphic forms of $S_1$ and $S_2$ respectively, $\gamma_{1}^1, \dots, \gamma_{22}^1\in H_2(S_1)$ and $\gamma_{1}^2, \dots, \gamma_{22}^2\in H_2(S_2)$ bases of cohomology.
Let $I_1$, $I_2$ be the intersection matrices in these bases, and $\pi_1 = (\int_{\gamma^1_j}\omega_1)_{1\le j\le22}$ and $\pi_2 = (\int_{\gamma^2_j}\omega_2)_{1\le j\le22}$ the row vectors of the periods of the holomorphic form.
Then a Hodge isometry is the data of a matrix $A\in GL_{22}(\Z)$ and a scalar $\lambda\in\C$ such that
\begin{equation}\label{eq:hodge_isom}
\pi_2 A = \lambda \pi_1 \text{ and } ^tAI_2A = I_1\,.
\end{equation}

Let $N_i\in \Z^{22\times 2}$ be the coordinate matrix of a basis of the Néron-Severi group $\NS(S_i)$
and $T_i\in \Z^{22\times 20}$ be the coordinate matrix of a basis of the transcendental lattice $\text{Tr}(S_i)$.
As these sublattices of $H_2(S_i)$ are algebraic invariants, we have the identities
\begin{equation}
AT_1 = T_2B \text{ and } AN_1 = N_2C
\end{equation}
for some invertible matrices $B\in GL_{20}(\Z)$ and $C\in GL_{2}(\Z)$.

Then $\lambda\pi_1T_1 = \pi_2AT_1 = \pi_2T_2B$. In particular, coefficient wise we have
\begin{equation}
\lambda (\pi_1T_1)_i = \sum_j(\pi_2T_2)_jB_{ji}\,,
\end{equation}
which allows us to recover the integers $B_{ji}$ using the LLL algorithm. We find that $\lambda=1$ with the choice $\omega_i = \res(\Omega/f_i)$.

Furthermore
\begin{equation}
^tC^tN_2I_2N_2C = ^tN_1^tAI_2AN_1 = ^tN_1I_1N_1\,.
\end{equation}
This yields 4 quadratic equations in the coefficients of $C$, to which we may find integer solutions. 
There are infinitely many solutions to this system, and not all yield a Hodge isometry -- they only do if the corresponding $A$ is an invertible integer matrix.

Indeed we have 
\begin{equation}
A
\left(\begin{array}{c|c}
&\\
\smash{\raisebox{.5\normalbaselineskip}{$T_1$}} &\smash{\raisebox{.5\normalbaselineskip}{$N_1$}}
\end{array}\right)
=
\left(\begin{array}{c|c}
&\\
\smash{\raisebox{.5\normalbaselineskip}{$T_2$}} &\smash{\raisebox{.5\normalbaselineskip}{$N_2$}}
\end{array}\right)
\left(\begin{array}{r|r}
B&0\\
\hline
0&C
\end{array}\right)\,,
\end{equation}
and thus
\begin{equation}
A
=
\left(\begin{array}{c|c}
&\\
\smash{\raisebox{.5\normalbaselineskip}{$T_2$}} &\smash{\raisebox{.5\normalbaselineskip}{$N_2$}}
\end{array}\right)
\left(\begin{array}{r|r}
B&0\\
\hline
0&C
\end{array}\right)
\left(\begin{array}{c|c}
&\\
\smash{\raisebox{.5\normalbaselineskip}{$T_1$}} &\smash{\raisebox{.5\normalbaselineskip}{$N_1$}}
\end{array}\right)^{-1}\in GL_{22}(\Z)\,.
\end{equation}
We pick solutions for $C$ and check whether they satisfy this condition.
We then verify that the conditions of \eqref{eq:hodge_isom} are also satisfied.
Using this method, we found that there is indeed a Hodge isometry between $S_1$ and $S_2$ up to high precision, which confirms that they are isomorphic.

\subsubsection{An explicit equation of a smooth quartic $K3$ surface with given Picard rank}\label{sec:picard_rank_examples}

\begin{table}[tp]
\begin{center}
\bgroup
\def\arraystretch{1.1}
\begin{tabular}{ c c }\toprule
\textbf{Defining polynomial} &\textbf{Picard number}\\
\midrule
$wx^3 + w^3y + y^4 + xz^3 + z^4 $&1\\
\thead{
$-x^{2} y^{2} + x y^{3} - y^{4} + x^{3} z + x^{2} y z - x y^{2} z - y^{3} z + x y z^{2} - y^{2} z^{2} + x z^{3} - y z^{3} + x^{3} w$\\
$- x^{2} y w - x y^{2} w - y^{2} z w - z^{3} w - x^{2} w^{2} - x y w^{2} + y^{2} w^{2} + y z w^{2} + y w^{3} - z w^{3}$
}&2\\
$x^4 - y^4 + z^4 - w^4 + (x-y)(z+w)y w - (x+y)(z-w)y^2$ &3\\
$x^3y + z^4 + y^3w + zw^3 $&4\\
\thead{
$5 x^4 + x^3 y - x y^3 - 5 y^4 + x^3 z + x^2 y z - x y^2 z - y^3 z + x^2 z^2 + 2 x y z^2 + 3 y^2 z^2 + 2 x z^3 + 2 y z^3$\\
$+ 2 z^4 - x^3 w + x^2 y w + x y^2 w - y^3 w + x^2 z w - y^2 z w - 2 x z^2 w + 2 y z^2 w + 2 z^3 w$\\
$- 2 x^2 w^2 - 2 x y w^2 - 2 y^2 w^2 - 2 x z w^2 - 2 y z w^2 + 2 z^2 w^2 + 2 x w^3 - 2 y w^3 - 2 z w^3 - 4 w^4$\\
}&5\\
$x^3y + y^4 + z^3w + yw^3 + zw^3$ &6\\
$w^3x + x^4 + wx^2z + x^3z + xy^2z - y^3z + wxz^2 + x^2z^2 - xz^3 + z^4$ &7\\
$x^3y + z^4 + y^3w + xw^3 + w^4$ &8\\
$w^4 + wx^2y + y^4 + x^3z - xy^2z + z^4$ &9\\
$x^3y + z^4 + y^3w + w^4$ &10\\
$w^4 + x^4 + x^2y^2 + y^4 - w^3z -^2 xy^2z + x^2z^2 + z^4$ &11\\
$x^3y + y^4 + z^3w + x^2w^2 + w^4 $&12\\
$w^4 +^3 x^4 + wy^3 + y^2z^2 + wz^3 +^2 xz^3$ &13\\
$x^3y + y^4 + z^3w + yw^3 + w^4 $&14\\
$x^3y + y^3z + z^4 + xy^2w + zw^3 $&15\\
$x^3y + y^4 + z^3w + xyw^2 + y^2w^2 + w^4 $&16\\
$x^3y + y^4 + z^4 + x^2w^2 + zw^3 $&17\\
$x^3y + x^3z + y^3z + yz^3 + w^4 $&18\\
$x^3y + z^4 + y^3w + xyzw + xw^3$ &19\\
$x^3y + z^4 + y^3w + xw^3 $&20\\\bottomrule
\end{tabular}
\egroup
\caption{Example polynomial for each Picard number. The new rows are Picard numbers 2, 3 and 5.}
\label{fig:example_table}
\end{center}
\end{table}

In this section we provide complements to Table 6.1 of~\textcite{LairezSertoz_2019}, for which examples of defining equations of quartic surfaces of Picard rank 2, 3 and 5 were missing.
Additionally, we have verified the known entries of this table using the method presented in this paper.

In addition to the example(s) of Section~\ref{sec:hodge_isometry}, Picard rank 2 smooth quartic $K3$ surfaces were found by testing generic polynomials with coefficients in $\{-1,0,1\}$. 
Following the construction of the proof of \textcite[Thm.~1.7]{Oguiso_2012}, we may construct an equation of a smooth quartic surface with Picard rank $3$.
We also stumbled upon a rank $5$ example, completing the missing entries.
The completed table can be found in Table~\ref{fig:example_table}.

Such polynomials are too large (in terms of the number of monomials) for the Picard rank to be recovered using the previous method of~\textcite{Sertoz_2019} with current numerical integration software. Using the SageMath implementation of the algorithm presented in this paper, we were able to numerically recover the Picard rank of these smooth quartic $K3$ surfaces in less than an hour each on a laptop.

\subsection{An example from Feynman integrals: the Tardigrade Graph}
\label{sec:an-example-from}
The methods presented here allowed the study of the geometry of a
parametrised Feynman integral corresponding to the Tardigrade
graph. The associated Feynman integral is
a  relative period 
integral in the sense of~\textcite{bek} and~\textcite{Brown:2015fyf}
for a family of singular quartic  surfaces defined by the Feynman graph
in~\figref{fig:tardigrade}, 
as argued in~\textcite{Bourjaily:2019hmc,Bourjaily:2018yfy} and
characterised  completely in~\textcite{DoranHarderPichonPharabodVanhove_2023}. 
This example is interesting because it goes beyond the scope of this paper, as the quartic surfaces associated to this graph are not smooth.
What's more, it shows how an approach relying on effective Picard-Lefschetz theory can be used to compute the periods of varieties given as elliptic fibrations.
Finally, it shows that our algorithm manages examples from physics that were previously out of reach.

\begin{figure}[h]
\begin{tikzpicture}[scale=0.6]
\filldraw [color = black, fill=none, very thick] (0,0) circle (2cm);
\draw [black,very thick] (-2,0) to (2,0);
\filldraw [black] (2,0) circle (2pt);
\filldraw [black] (0,2) circle (2pt);
\filldraw [black] (0,-2) circle (2pt);
\filldraw [black] (0,0) circle (2pt);
\filldraw [black] (-2,0) circle (2pt);
\draw [black,very thick] (-2,0) to (-3,0);
\draw [black,very thick] (2,0) to (3,0);
\draw [black,very thick] (0,2) to (0,3);
\draw [black,very thick] (0,-2) to (0,-3);
\draw [black,very thick] (0,0) to (0,-1);
\end{tikzpicture}
\caption{The tardigrade graph}\label{fig:tardigrade}
\end{figure}

The Tardigrade graph corresponds to a family of $K3$ surfaces given as the minimal resolution of a family of generically singular quartic surfaces in $\Proj^3$. 
For example, one element of this family, on which the computation was
performed, is the quartic defined by the equation derived in~\textcite[\S8]{DoranHarderPichonPharabodVanhove_2023}
\begin{equation}
\begin{split}
&6124 x^{4} - 24782692 x^{3} x + 24962401977 x^{2} x^{2} - 20842243972 x x^{3} + 4331388844 x^{4} + 6124 w^{4} \\
&\hspace{1em}+ 13827992 x^{3} z - 27919677996 x^{2} x z - 119291704836 x x^{2} z + 50444249752 x^{3} z \\
&\hspace{1em}+ 7840306116 x^{2} z^{2}- 1895725740 x x z^{2} + 168749562396 x^{2} z^{2} + 38842829528 x z^{3}\\
&\hspace{1em} + 168487393048 x z^{3} + 48321305644 z^{4} - 101996680 x^{3} w + 204653103868 x^{2} x w\\
&\hspace{1em} - 85803990572 x x^{2} w + 10300568 x^{3} w - 115176640844 x^{2} z w- 460049503942 x x z w\\
&\hspace{1em} - 18769272012 x^{2} z w - 311785995116 x z^{2} w - 108990818964 x z^{2} w - 62891049316 z^{3} w\\
&\hspace{1em}+ 417760330428 x^{2} w^{2} - 126779372 x x w^{2} + 10306692 x^{2} w^{2} + 179497287052 x z w^{2}\\
&\hspace{1em} + 37592148 x z w^{2}+ 22845754953 z^{2} w^{2} - 101996680 x w^{3} + 12248 x w^{3} - 22389124 z w^{3}\,.
\end{split}
\end{equation}

Using a variation on the methods presented in this paper, of which we give an overview below, we were able to recover algebraic invariants of this family, namely
\begin{enumerate}
\item its generic Picard rank, 11, 
\item the Picard lattice, i.e.\ the kernel of the holomorphic period map $\gamma\mapsto \int_\gamma \omega$,
\item and an embedding of the Picard lattice in the standard $K3$ lattice.
\end{enumerate}

The quartic surfaces considered generically have 4 nodal singularities. 
We want to obtain the periods of their minimal desingularisation.
Despite the presence of singularities, we may still consider a Lefschetz fibration of these varieties and compute the monodromy matrices around the critical values.
From these monodromy matrices, we recover a fibration of a formal smoothing of the singular quartic surface which yields a description of its homology.
The homology of the smoothing of the singular quartic is isometric to that of its desingularisation.
Therefore this allows us to compute a description of the homology of the $K3$ surface. 
We can then integrate the holomorphic form following Section~\ref{sec:comp-periods}.
This shows that we may use the effective Picard-Lefschetz theory presented in this paper to recover the homology of hypersurfaces with nodal singularities.

In practice, however, in an effort to reduce the runtime, we may instead consider an elliptic fibration of the $K3$ surface directly.
This has two main benefits: 
\begin{itemize}
\item The order of the Picard-Fuchs equations that need to be integrated diminishes from $7$ to $3$ (as the genus of the homology of the fibre goes from $3$ for a quartic curve to $1$ for an elliptic curve). 
This greatly decreases the runtime of the computation down to less than a minute.
\item We do not need to consider a modification of the $K3$ surface, which means there are no superfluous exceptional divisors to remove.
\end{itemize}
In order to obtain an elliptic fibration of the $K3$ surface, we project the $K3$ surface away from one of its singular points. 
This yields a double cover of $\Proj^2$ ramified along a sextic curve. 
Up to a change of variable, the defining equation of this sextic curve can be made quartic in one of the variables. 
Taking another variable as a parameter, we obtain an affine equation of the form $y^2 = p(x,t)$ with $p$ quartic in $x$, which gives an elliptic fibration of the $K3$ surface parametrised by $t$.
This fibration has 17 singular fibres, two of which are $I_4$ fibres, one is an $I_2$ and the rest are $I_1$'s as per the Kodaira classification of singular fibers of elliptic surfaces.
We identify the type of the singular fibres by looking at their monodromy matrices.
Passing again to a smoothing of the fibration, we are able to recover the homology and perform the integration.
Further details about this computation are given by \textcite[App.~B]{DoranHarderPichonPharabodVanhove_2023}.

\printbibliography

\end{document}